\documentclass[11pt]{amsart}
\usepackage[margin=1in]{geometry}
\usepackage{graphicx}
\usepackage{amsmath}
\usepackage{amsthm,amsfonts,amssymb,mathrsfs,amscd,amstext,amsbsy}
\usepackage{epic,eepic}
\usepackage{yfonts}
\usepackage{paralist,enumerate}
\usepackage[all]{xy}
\usepackage{hyperref}

\hypersetup{colorlinks}

\newtheorem{theorem}{Theorem}[section]
\newtheorem{corollary}[theorem]{Corollary}
\newtheorem{lemma}[theorem]{Lemma}
\newtheorem{proposition}[theorem]{Proposition}

\newtheorem{question}[theorem]{Question}
\theoremstyle{definition}
\newtheorem{definition}[theorem]{Definition}

\newtheorem{remark}[theorem]{Remark}

\makeatletter
\renewenvironment{proof}[1][\proofname]{%
   \par\pushQED{\qed}\normalfont%
   \topsep6\p@\@plus6\p@\relax
   \trivlist\item[\hskip\labelsep\bfseries#1\@addpunct{.}]%
   \ignorespaces
}{%
   \popQED\endtrivlist\@endpefalse
}
\makeatother


\newtheorem{theore}{Theorem}[section]

\newtheorem{corollar}{Corollary}

\addtocounter{corollar}{3}

\newtheorem{Definition}[theorem]{Definition}
\newtheorem*{Definition*}{Definition}

\newcommand\blfootnote[1]{%
  \begingroup
  \renewcommand\thefootnote{}\footnote{#1}%
  \addtocounter{footnote}{-1}%
  \endgroup }

\def\qed{\hfill \ifhmode\unskip\nobreak\fi\quad\ifmmode\Box\else$\Box$\fi\\ }

\begin{document}

\title[Algorithms detecting stability and Morseness]{Algorithms detecting stability and Morseness for finitely generated groups }
\author{Heejoung Kim}
\blfootnote{\textit {Date}: August 13, 2019.
\indent
\thanks{}} 

\address{Department of Mathematics, University of Illinois at Urbana-Champaign, Urbana, IL 61801}
\email{hkim404@illinois.edu}

\begin{abstract}

The notions of stable and Morse subgroups of finitely generated groups generalize the concept of a quasiconvex subgroup of a word-hyperbolic group. For a word-hyperbolic group $G$, Kapovich \cite{K96} provided a partial algorithm  which, on input a finite set $S$ of $G$, halts if $S$ generates a quasiconvex subgroup of $G$ and runs forever otherwise.
In this paper, we give various detection and decidability algorithms for stability and Morseness of a finitely generated subgroup of mapping class groups, right-angled Artin groups, toral relatively hyperbolic groups, and  finitely generated groups discriminated by a locally quasiconvex torsion-free hyperbolic group (for example, ordinary limit groups).
\end{abstract}

\maketitle

\section{Introduction}

For a word-hyperbolic group $G$, a subgroup $H$ is called \textit{quasiconvex} in $G$ if for some (equivalently, for any) finite generating set S of G there exists a uniform constant $N\geq 0$ such that every geodesic in the Cayley graph $\Gamma(G,S)$ of $G$ with respect to $S$ that connects a pair of points in $H$ is contained in the $N$-neighborhood of $H$. 
Quasiconvex subgroups play a significant role in the theory of word-hyperbolic groups where quasiconvexity is equivalent to being finitely generated and undistorted.
Kapovich \cite{K96}  provided  a partial algorithm detecting quasiconvexity of a finitely generated subgroup of a word-hyperbolic group. See also \cite{KMW17}.

\begin{proposition}[Proposition 4 in \cite{K96}]\label{K96}
Let $G$ be a word-hyperbolic group given by a finite presentation $G=\langle x_1,\dots, x_n\,|\,r_1,\dots, r_m \rangle$ and let $S = \{x_1^{\pm 1},\dots , x_n^{\pm 1}\}$.
Then there is a uniform algorithm which, given a finite set of words $v_1,..,v_t$ over $S$, will
\begin{itemize}
\item eventually stop and produce the quasiconvexity constant $\epsilon$ and the distortion constant $C$ of the subgroup $H=gp(\bar{v_1}, \dots, \bar{v_t})$ of $G$ if $H$ is quasiconvex in $G$.
\item run forever if $H$ is not quasiconvex in $G$.
\end{itemize}
\end{proposition}

Since the above algorithm does not detect non-quasiconvex subgroups, one might ask the following question.

\begin{question}\label{q0}
Given a word-hyperbolic group $G$, is there an algorithm that, for a finite subset $S$ of $G$, decides whether or not $H=\langle S\rangle$ is quasiconvex in G?
\end{question}

In general, the answer is no, see  \cite{BW01}. 
For an arbitrary finitely generated group, there are two recently introduced generalizations  of the notion of a quasiconvex subgroup of a word-hyperbolic group, namely a \textit{stable} subgroup and a \textit{Morse} (or \textit{strongly quasiconvex}) subgroup.

\begin{Definition}[Stable subgroups \cite{DT15}]
Let $G$ be a finitely generated group and let $H$ be a finitely generated subgroup of $G$. We say that $H$ is \textit{stable} in $G$ if $H$ is undistoreted in $G$ and if for every (equivalently, some) finite generating set $S$ of $G$ and for every $k\geq1$ and $c\geq0$ there is some $L=L(S,k,c)$ such that for every pair of $(k,c)-$quasigeodesics in $G$ with the same endpoints on $H$, each of these two qausigeodesics is contained in the $L$-neighborhood of the other.
\end{Definition}

\begin{Definition}[Morse subgroups \cite{T17, G17}]
\textup{Let $G$ be a finitely generated group and let $H$ be a subgroup of $G$. We say that $H$ is a \textit{Morse} subgroup of $G$ if for every (equivalently, some) finite generating set $S$ of $G$, for every $k\geq1$ and $c\geq0$ there is some $M=M(k,c)$ such that every $(k,c)-$quasigeodesic in $G$ with endpoints on $H$ is contained in the $M$-neighborhood of $H$.}
\end{Definition}

Being stable implies being Morse, but in general stability and Morseness are not equivalent.
A stable subgroup $H$ of  a finitely generated group $G$ is necessarily word-hyperbolic but every subgroup $H$ of finite index in G is Morse. For example, $\mathbb{Z}\times\mathbb{Z}$ is a non-stable Morse subgroup of itself. Tran \cite{T17} proved that a finitely generated subgroup $H$ of a finitely generated group $G$ is stable if and only if $H$ is hyperbolic and Morse in $G$. Motivated by Proposition \ref{K96} and Question \ref{q0}, we are interested in the following questions.

\begin{question}\label{q}
Let $G$ be a finitely generated group from a particular class of groups. Is there a partial algorithm that, for a finite subset $S$ of $G$, is guaranteed to terminate in finite time and output the answer whether $H=\langle S\rangle$ is a stable subgroup of $G$ but may run forever if $H$ is not stable?
Is there a complete algorithm that decides whether or not a finitely generated subgroup $H$ of $G$ is stable?
\end{question}

\begin{question}\label{q1}
What if we replace stable by Morse in Question \ref{q}?
\end{question}

The first main theorem answers to the above questions for mapping class groups as follows.

\begin{theore}\label{main1}
\textit{
Let $S$ be an oriented, connected, finite type surface with $\chi(S)<0$ which is neither the 1-punctured torus nor the 4-punctured sphere and let \textup{Mod$(S)$} be the mapping class group of $S$. }

\begin{enumerate}[{(i)}]
\item\label{main1-1}\textit{There is a partial algorithm which,  for a subgroup $H$ of  \textup{Mod$(S)$} given by a finite generating set, will terminate if $H$ is stable in \textup{Mod$(S)$} and run forever if $H$ is not stable in \textup{Mod$(S)$.} }

\item\label{main1-2}\textit{There is a complete algorithm which, for an undistorted subgroup $H$ of \textup{Mod$(S)$}, decides whether or not $H$ is stable.}

\item\label{main1-3}\textit{There is a partial algorithm which, for a finitely generated subgroup $H$ of  \textup{Mod$(S)$} given by a finite generating set, will terminate if $H$ is Morse in \textup{Mod$(S)$} and run forever if $H$ is not Morse in \textup{Mod$(S)$.} }

\end{enumerate}
\end{theore}

A finitely generated subgroup $H$ of Mod($S$) is stable if and only if its orbit into the curve graph $\mathcal{C}(S)$ is a quasi-isometrically embedding. It is known that the distances in $\mathcal{C}(S)$ can be computed algorithmically \cite{B06}.
For part (\ref{main1-1}) of Theorem \ref{main1}, we use the ``local-to-global" principle to detect subgroups $H$ of Mod$(S)$ with quasi-isometrically embedded orbits in $\mathcal{C}(S)$.
For an undistorted subgroup $H$ of Mod($S$), it is known that $H$ is stable in Mod($S$) if and only if $H$ is purely loxodromic \cite{DT15, BKKL18}. Therefore, for the proof of (\ref{main1-2}), we run a partial algorithm in (\ref{main1-1}) and in parallel, look for a non-loxodromic element in $H$.
It is known that a finitely generated subgroup $H$ of Mod$(S)$ is Morse if and only if $H$ is stable or has finite index in Mod($S$) \cite{K19}. For part (\ref{main1-3}), we run a partial algorithm in (\ref{main1-1}) and in parallel, run the Todd-Coxeter algorithm to detect finite index of a finitely generated subgroup.

The second main theorem answers to Question \ref{q} and Question \ref{q1} for right-angled Artin groups.

\begin{theore}\label{main2}
\textit{Let $\Gamma$ be a finite connected and anti-connected graph with at least two vertices. }
\begin{enumerate}[{(i)}]
\item\label{main2-1}\textit{There is a complete algorithm which, for a subgroup $H$ of $A_\Gamma$ given by a finite generating set, will terminate and determine whether or not $H$ is stable in $A_\Gamma$.}
\item\label{main2-2}\textit{There is a partial algorithm which, for a subgroup $H$ of $A_\Gamma$ given by a finite generating set, will terminate if $H$ is Morse in $A_\Gamma$ and run forever if $H$ is not Morse in $A_\Gamma$. }
\end{enumerate}
\end{theore}

We provide two different algorithms for Theorem \ref{main2}(\ref{main2-1}).
The first algorithm uses the \textit{extension graph} $\Gamma ^e$ \cite{KK13,KK14} and ``star metric'' on $A_\Gamma$ which is quasi-isometric to  $\Gamma ^e$ and comparable with the standard normal form. 
A finitely generated subgroup $H$ of $A_\Gamma$ is stable if and only if its orbit into $\Gamma^e$ with the graph metric is a quasi-isometrically embedding if and only if $H$ is purely loxodromic.
Thus, the first algorithm for Theorem \ref{main2}(\ref{main2-1}) proceeds by running a partial algorithm detecting quasiconvexity of an orbit $H$ in $\Gamma^e$ and in parallel, iterating elements to check that there is a non-loxodromic element to detect non-stability of $H$.
For the second algorithm for Theorem \ref{main2}(\ref{main2-1}), we look for a cube complex in \cite{MT16} encoding all infinite order elements in $H$ and check whether or not a closed loop labeled by a join word, i.e., non-loxodromic element.
Note that under the assumptions of Theorem \ref{main2}, a subgroup $H$ is Morse if and only if  either $H$ is stable or $H$ has finite index in $G$ \cite{T17,G17}.
For part (\ref{main2-2}), we run a partial algorithm in (\ref{main2-1}) and in parallel, run the Todd-Coxeter algorithm to detect finite index of a finitely generated subgroup.

We now consider toral relatively hyperbolic groups for Question \ref{q} and Question \ref{q1}.
A finitely generated group $G$ is called a \textit{toral relatively hyperbolic group} if $G$ is torsion-free and hyperbolic relative to a (possibly empty) finite collection $\mathbb{P}$ of finitely generated free abelian non-cyclic subgroups of $G$.

\begin{theore}\label{main3}
\textit{Let $(G, \mathbb{P})$ be a toral relatively hyperbolic group. }
\begin{enumerate}[{(i)}]
\item\label{main3-1}\textit{There is a partial algorithm which, for a subgroup $H$ of $G$ given by a finite generating set, will terminate if $H$ is stable in $G$ and run forever if $H$ is not stable in $G$.}

\item\label{main3-2}\textit{There is a complete algorithm which, for an undistorted subgroup $H$ of $G$, decides whether or not $H$ is stable.}

\item\label{main3-3}\textit{There is a partial algorithm which, for a subgroup $H$ of $G$ given by a finite generating set, will terminate if $H$ is Morse in $G$ and run forever if $H$ is not Morse in $G$. }

\item\label{main3-4}\textit{There is a complete algorithm which, for an undistorted finitely generated subgroup $H$ of $G$, decides whether or not $H$ is Morse.}
\end{enumerate}

\end{theore}

Tran \cite{T17} gave complete characterizations of stability and Morseness of an undistorted subgroup $H$ of a relatively hyperbolic group $(G, \mathbb{P})$. These characterizations involve checking the properties of all infinite intersections $H\cap P^g$, where $g\in G$ and $P\in \mathbb{P}$.
For part (\ref{main3-1}) of Theorem \ref{main3}, we combine these results of Tran with recent algorithmic results of Kharlampovich, Myasnikov and Weil \cite{KMW17} about toral relatively hyperbolic groups.
When an undistorted subgroup is given as (\ref{main3-2}), we run the algorithm in (\ref{main3-1}) and in parallel, detect non-stability by Tran's characterization of stability. The approach to (\ref{main3-3}) is similar. For part (\ref{main3-4}), we run the algorithm in part (\ref{main3-3}) to detect Morseness and in parallel, run a partial algorithm detecting non-Morseness by using relatively hyperbolic Dehn fillings \cite{O07,GM08, GM17}. 
Specifically, we use the results of Groves and Manning \cite{GM17} on the behavior of relatively quasiconvex subgroups under Dehn fillings.
Producing an algorithm for detecting non-Morseness of undistorted finitely generated subgroups in $G$  is the most  involved portion of the proof of Theorem \ref{main3}(\ref{main3-4}) since it requires iteratively applying the above procedures to groups obtained from $G$ by hyperbolic Dehn fillings.

Note that a new result of Kharlampovich and Weil (Theorem 2 in \cite{KW19}) implies that if $(G,\mathbb{P})$ is a toral relatively hyperbolic group then there is an algorithm which, for given $g, h_1,\dots ,h_n\in G$, decides whether or not $g\in H=\langle h_1,\dots ,h_n\rangle$, assuming that $H$ is relatively quasiconvex. 
This result is related to, but does not imply, our Theorem \ref{main3}(\ref{main3-4}). The proof of Theorem 2 in \cite{KW19} utilizes an element-wise separability result of Manning and Martinez-Perdoza \cite{MM10}. That theorem implies that if $H$ is relatively quasiconvex in a toral relatively hyperbolic group $G$ and if $g\in G\setminus H$, then there exists a Morse subgroup $H_1\le G$ such that  $H\le H_1$ and $g\not\in H_1$.

Theorem  \ref{main3}(\ref{main3-4}) also has the following useful corollary:

\begin{corollar}\label{main5}
Let $(G,\mathbb{P})$ be a toral relatively hyperbolic group. Then there exists an algorithm that, given an undistorted finitely generated subgroup $H$ of $G$, decides whether or not $H$ has finite index in $G$.
\end{corollar}

Let $G$ be a finitely generated group discriminated by a locally quasiconvex torsion-free hyperbolic group. Then $G$ is a toral relatively hyperbolic group and every finitely generated subgroup $H$ of $G$ is undistorted (see Lemma 4.6 below). Thus, Theorem \ref{main3}(\ref{main3-2}) and Theorem \ref{main3}(\ref{main3-4}) imply the following corollary.

\begin{corollar}\label{main4}
\textit{
Let $G$ be a finitely generated group discriminated by a locally quasiconvex torsion-free hyperbolic group.}
\begin{enumerate}[{(i)}]
\item\label{main4-1}\textit{There is a complete algorithm which, for a subgroup $H$ of $G$ given by a finite generating set, decides whether or not $H$ is stable.}
\item\label{main4-2}\textit{There is a complete algorithm which, for a subgroup $H$ of $G$ given by a finite generating set, decides whether or not $H$ is Morse.}
\end{enumerate}

\end{corollar}

Note that since an ordinary limit group $G$ is a finitely generated group discriminated by the free group $F_2$, and $F_2$ is locally quasiconvex torsion-free hyperbolic, there exist such algorithms for $G$ as in Corollary \ref{main4}.
\bigskip

The above results raise several interesting questions.
In Theorem \ref{main1}, if we can detect distortion of $H$ in Mod($S$) we have a complete algorithm detecting stability by the equivalence (3) and (4) in Theorem \ref{convex}. Hence, it is natural to ask the following question.

\begin{question}
 Is there a patrial algorithm detecting that a finitely generated subgroup is undistorted in Mod($S$)?
\end{question}

Let $G$ be either a mapping class group as in Theorem \ref{main1} or a right-angled Artin group as in Theorem \ref{main2}. Then we do not have a partial detection algorithm for non-Morseness of a given undistorted subgroup $H$ of $G$. One way to detect non-Moreseness is checking non-stability and infinite index. Since we already have a partial algorithm for detecting non-stability, the question is as follows.

\begin{question}
Let $G$ be a mapping class group or a right-angled Artin group. Is there a patrial algorithm which, for an undistorted subgroup $H$, decides whether or not $H$ has infinite index in $G$?
\end{question}

The paper is organized as follows. 
We prove Theorem \ref{main1} in Section \ref{2} and Theorem \ref{main2} in Section \ref{3}.
In Section \ref{4}, we prove Theorem \ref{main3} and Corollary \ref{main5}.  Finally, we prove  Corollary \ref{main4} in Section \ref{6}.
\\

\textbf{Acknowledgement. }The author is very grateful to her  PhD advisor Ilya Kapovich for his guidance, encouragement, and constant support. Ilya Kapovich and the author would like to thank to Samuel Taylor for suggesting an approach to one of the proofs of Theorem \ref{main2}(\ref{main2-1}). 
We would also like to thank Daniel Groves, Chris Hruska and Olga Kharlampovich for useful discussions regarding relatively quasiconvex subgroups of relatively hyperbolic groups. We are particularly grateful to Daniel Groves for explaining to us the results of \cite{GM17}.
The author gratefully acknowledges support from the NSF grants DMS-1710868 and DMS-1905641.

\section{Mapping class groups}\label{2}

Let $S$ be an oriented, connected, finite type surface with Euler characteristic $\chi(S)<0$. Note that a surface
$S$ is of finite type if and only if the fundamental group of $S$ is finitely generated. The (extended) mapping class group Mod($S$) of $S$ is the group of isotopy classes of homeomorphisms of $S$. Convex cocompact subgroups of Mod($S$) provide an important class of subgroups of Mod($S$), see \cite{FM02, KL08}.

\begin{definition}\label{definitionconvex}
Let $S$ be an oriented, connected, finite type surface with $\chi(S)<0$. Then a subgroup $H<$ Mod$(S)$ is \textit{convex cocompact} if for some $x$ in Teichm$\ddot{\textup{u}}$ller space $\mathcal{T}(S)$ the orbit $H \cdot x$ is quasiconvex with respect to the Teichm$\ddot{\textup{u}}$ller metric on $\mathcal{T}(S)$.
\end{definition}

We have the following various characterizations of convex cocompactness of Mod($S$).

\begin{theorem}\label{convex}
Let $S$ be an oriented, connected, finite type surface with $\chi(S)<0$ which is neither the 1-punctured torus nor the 4-punctured sphere.
Let \textup{Mod$(S)$} be the mapping class group of $S$, and let $H$ be a finitely generated subgroup of \textup{Mod$(S)$}. Then the following are equivalent:
\begin{enumerate}
\item The subgroup $H$ is convex cocompact.
\item An orbit map of $H$ into the curve complex $\mathcal{C}(S)$ is a quasi-isometric embedding.
\item The subgroup $H$ is finitely generated, undistorted, and purely pseudo-Anosov.
\item The subgroup $H$ is stable.
\item The subgroup $H$ is Morse of infinite index in \textup{Mod$(S)$}.
 \end{enumerate}
\end{theorem}

The equivalence between (1) and (2) was shown in \cite{KL08} and Hamenst$\ddot{\textup{a}}$dt in \cite{H05} independently. The equivalence between (1) and (3) was proved in \cite{BKKL18} including the 1-punctured torus $S_{1,1}$ and the 4-punctured sphere $S_{0,4}$. The equivalence between (1) and (4) shown in \cite{DT15} excludes those two surfaces. The equivalence between (4) and (5) shown in \cite{K19, RST18}.

\begin{theorem}\cite{K19}\label{K19} Let $S$ be an oriented, connected, finite type surface with $\chi(S)<0$ which is neither the 1-punctured torus nor the 4-punctured sphere. A finitely generated subgroup $H$ of \textup{Mod$(S)$} if and only if either (i) $H$ is stable in \textup{Mod$(S)$} or (ii) $H$ has finite index in \textup{Mod$(S)$}, and that (i) and (ii) are mutually exclusive.
\end{theorem}

\subsection{Curve graph}\label{curve complex}

Recall that a geodesic metric space $X$ is $\delta$-\textit{hyperbolic} where $\delta \geq 0$, if for any geodesic triangle $T$ in $X$ each side of $T$ is contained in the $\delta$-neighborhood of the union of two other sides. 
We also recall the following well-known fact about $\delta$-hyperbolic spaces.

\begin{lemma}[Local to global principle]\label{lem}
Let $X$ be a $\delta$-hyperbolic geodesic space.
For all integers $\lambda\geq 1, $ and $\epsilon\geq 0$ there exists $K, \lambda',$ and $ \epsilon '$ such that, if every length $K$ segment of $\gamma$ is a $(\lambda,\epsilon)$-quasigeodesic, then $\gamma$ is a ($\lambda',\epsilon')$-quasigeodesic. Moreover, the constants $K, \lambda',$ and $ \epsilon'$ can be computed algorithmically from $\delta,\lambda,$ and $\epsilon$.
\end{lemma}

For the proof of Lemma \ref{lem}, for example, see Chapter III in \cite{BH99}.
Let $S$ be a surface as in Definition \ref{definitionconvex}.
The \textit{curve complex} is a natural simplicial complex associated to $S$, and Mod($S$) acts on the curve complex by simplicial automorphisms. We utilize \textit{curve graph} $ \mathcal{C}(S)$, one-skeleton of the curve complex, whose vertices are isotopy classes of essential simple closed curves on $S$, and two distinct isotopy classes are joined by an edge if they are disjointly realizable.
It is known that the curve graph $\mathcal{C}(S)$ is a $\delta$-hyperbolic space for some $\delta$ which can be chosen independent of $S$, see \cite{B06, H07,A13, B12, CRS14, HPW15}. 
Leasure \cite{L02} found an algorithm to compute the distance between two vertices of $\mathcal{C}(S)$, and after then other algorithms have been produced by Shackleton \cite{S12}, Webb \cite{W15}, Watanabe \cite{W17}, and Birman, Margalit, and Menasco \cite{BMM16}.  Bowditch \cite{B06} provided an algorithm to compute geodesics in $\mathcal{C}(S)$. The first algorithm for Theorem \ref{main1}(\ref{main1-1}) is as follows.

\begin{proof}[First proof of Theorem \ref{main1}(\ref{main1-1})]

Pick a base vertex $x$ in the curve graph $\mathcal{C}(S)$. 
Suppose that $H$ is a subgroup of Mod($S$) generated by a finite set $A$ in Mod($S$) such that $A=A^{-1}$.
For each $a\in A$, by using Bowditch's algorithm, compute a geodesic starting from $x$ and terminating at $a\cdot x$ in the cure graph $\mathcal{C}(S)$ and denote it as $u_a$.
If $a\cdot x=x$, we choose $u_a$ to be a closed edge-path of length $2$ from $x$ to $x$ in $\mathcal{C}(S)$.  
 Let $f:$ Mod($S$)$ \to \mathcal{C}(S)$ be an orbit map sending $\phi$ to $\phi \cdot x$ for every $\phi \in\, $Mod($S$). We can extend the orbit map $f$ to the Cayley graph $\Gamma_{\textup{Mod($S$)}}$ of Mod($S$) by sending an edge $(\phi, \phi a)$ to the path $\phi \cdot u_a$ from $\phi \cdot x$ to $(\phi a)\cdot x$, where $a\in A$. 
Note that, by construction, $f$ sends an edge path of length $n$ in $\Gamma_{\textup{Mod}(S)}$ to an edge-path of length at least $n$ in $\mathcal{C}(S)$. 
By Theorem \ref{convex}, the subgroup $H$ is stable in Mod($S$) if and only if there exist integers $\lambda\geq 1, \epsilon \geq 0$ such that for every geodesic $w$ in $\Gamma_{\textup{Mod($S$)}}$ the image $f(w)$ is a ($\lambda$,$ \epsilon$)-quasigeodesic in $\mathcal{C}(S)$.
This is also equivalent to the existence of $\lambda$ and $\epsilon$ such that for every geodesic $w$ starting from $1$ in $\Gamma_{\textup{Mod($S$)}}$ the image $f(w)$ is a ($\lambda$, $ \epsilon$)-quasigeodesic in $\mathcal{C}(S)$.
The algorithm for detecting stability proceeds as follows. 

Start enumerating all pairs of integers $\lambda\geq 1, \epsilon \geq 0$.
For every such pair $\lambda, \epsilon$, compute the constants $K$, $\lambda' \geq 1$ and $\epsilon'\geq 0$ from Lemma \ref{lem}. Using the solution of the word problem in Mod($S$), list all the geodesic edge-paths $w$ in $\Gamma_{\textup{Mod($S$)}}$ of length at most $K$. 
For every such $w$, compute a path $f(w)$ and check whether or not $f(w)$ is a ($\lambda', \epsilon')$-quasigeodesic in $\mathcal{C}(S)$.
 If the answer is `yes' for all such $w$, we terminate the algorithm. Otherwise we proceed to the next pair of $\lambda$ and $\epsilon$. 
 
 This algorithm terminates if and only if $H$ is stable. \end{proof}

\begin{proof}[Proof of Theorem \ref{main1}(\ref{main1-2})]
Recall that an element $\phi\in Mod(S)$ acts loxodromically on $\mathcal{C}(S)$  if and only if $\phi$ is pseudo-Anosov. Recall also that there exists an algorithm that, given $\phi\in Mod(S)$, decides whether or not $\phi$ is pseudo-Anosov \cite{BH95}.

For a given undistorted finitely generated subgroup $H$ of Mod($S$), we run the partial algorithm in Theorem \ref{main1}(\ref{main1-1}) and, in parallel, we start enumerating all elements of $H$ and look for a non-loxodormic element of $H$.
By Theorem \ref{convex}, either $H$ is stable or else $H$ contains a non-loxodromic element. If the algorithm in Theorem \ref{main1}(\ref{main1-1}) terminates on $H$, we declare that $H$ is stable in Mod($S$). If we find a non-loxodromic element in $H$, we declare that $H$ is not stable in Mod($S$).  

This procedure produces a complete algorithm for deciding whether or not $H$ is stable in Mod($S$).\end{proof}

\begin{proof}[Proof of Theorem \ref{main1}(\ref{main1-3})]
 Let $H$ be a finitely generated subgroup of Mod($S$) given by a finite generating set. By Theorem \ref{K19}, the subgroup $H$ is Morse in Mod($S$) if and only if either $H$ is stable in Mod($S$) or $H$ has finite index in Mod($S$). 
 
We now run the partial algorithm for detecting Morseness in Theorem \ref{main1}(\ref{main1-1}), and, in parallel, run the Todd-Coxeter coset enumeration algorithm (see Chapter III, Section 12 in \cite{RP77}) for detecting finiteness of the index of $H$ in Mod($S$). If the algorithm in Theorem \ref{main1}(\ref{main1-1}) terminates with the conclusion that $H$ is stable, we declare that $H$ is Morse in Mod($S$). If $H$ has finite index in Mod($S$), the Todd-Coxeter algorithm eventually terminates and discovers this fact. In that case, we declare that $H$ is Morse in Mod($S$). Otherwise we continue running both the algorithm in Theorem \ref{main1}(\ref{main1-1}) and the Todd-Coxeter algorithm run forever. 

Taken together, this procedure provides a partial algorithm for detecting Morseness of H in Mod($S$), as required.  
\end{proof}

\begin{remark}
At the moment, it is not known if there exists a complete algorithm that, given an undistorted finitely generated subgroup $H$ of Mod($S$), decides whether or not $H$ has finite index in Mod($S$). If such an algorithm is found, we can promote Theorem \ref{main1}(\ref{main1-3}) to a complete algorithm for deciding whether or not an undistorted subgroup of Mod($S$) is Morse.  
\end{remark}

In Section \ref{short exact sequence}, we provide another algorithm for Theorem \ref{main1}(\ref{main1-1}) when $S$ is a closed hyperbolic surface.

\subsection{Short exact sequence}\label{short exact sequence}
In this section, we assume that $S$ is a closed hyperbolic surface. 
For a finitely generated subgroup $H$ of Mod($S$) the \textit{extension group} $ E_H$ is obtained from the short exact sequence $1 \rightarrow \pi_1(S) \rightarrow E_H \rightarrow H \rightarrow 1$ induced by the Birman exact sequence $1 \rightarrow \pi_1(S) \rightarrow \textup{Mod}(S\setminus \{p\}) \rightarrow \textup{Mod}(S) \rightarrow 1$ for a point $p\in S$.
For a closed hyperbolic surface $S$, it is known that $E_H$ is hyperbolic if and only if $H$ is convex cocompact in Mod($S$) \cite{FM02, H05}. Suppose that we know a finite presentation of $H$. Then we can provide a presentation for $ E_H$ algorithmically as follows.

\begin{lemma}\label{presentation}
Let $H$ be a finitely presented subgroup of \textup{Mod$(S)$} for a closed hyperbolic surface $S$. Then the extension group $E_H$ is finitely presented. Moreover, we can algorithmically find a finite presentation for $E_H$, given a finite generating set $Y\subseteq \textup{Mod(}$S$ \textup {)}$ for $H$ and a finite presentation $H=\langle Y\,|\,Z\rangle$ for $H$. 

\end{lemma}

\begin{proof}
Let $\pi_1(S)=\langle X \,|\, R\rangle$ be a finite presentation of $\pi_1(S)$ and let $H=\langle Y\,|\,Z\rangle$ be a finitely presented subgroup of Mod($S$). By the exactness of the sequence $1 \rightarrow \pi_1(S) \rightarrow E_H \rightarrow H \rightarrow 1$, the fundamental group $\pi_1(S)$ of $S$ can be identified with a normal subgroup of $E_H$ and the group homomorphsim $\pi : E_H\rightarrow H$ is surjective.
Pick a lift $\phi: Y\to E_H$ of the quotient map $\pi: E_H\to H$. For each $y\in Y$ denote $y'=\phi(y)$, and for every $z=y_1^{\epsilon_1}\dots y_n^{\epsilon_n}\in Z$ denote $z'=(y_1')^{\epsilon_1}\dots (y_n')^{\epsilon_n}$.
Let $A := \{ y', x \,|\, y\in Y, x \in X \}$. Since $\pi_1(S)=Ker(\pi)$ is normal in $E_H$, for every $x\in X, y\in Y, \epsilon \in \{\pm 1\}$, there exists a word $u=u_{x,y, \epsilon}\in F(X)$ such that  $(y')^{\epsilon} x(y')^{-\epsilon}=u$ in $E_H$. Moreover, such $u$ can be found algorithmically as follows. Start enumerating all words from $F(X)$ and, using the solution of the word problem in Mod($S$), for each of these words check whether it is equal to $(y')^{\epsilon} x(y')^{-\epsilon}$ in $E_H$. Eventually this process stops and produces a desired word $u_{x,y,\epsilon}$. For every $z\in Z$ there exists some word $v=v_z\in F(X) $ such that $z'=v$ in $E_H$ and we can find such $v$  algorithmically in a similar manner. 
 By this procedure, we can compute the following finite set $U$ algorithmically:
 \centerline{$U:= \{r,\,  (y')^{\epsilon} x(y')^{-\epsilon},\, z'v_z^{-1} \,|\, r\in R,\,  x\in X, y\in Y, \epsilon \in \{\pm 1\}, z\in Z
 \}.$}
 \noindent
Note that the set $U$ is contained in $F(A)$.

\begin{enumerate}[{(i)}]
\item Claim 1 : The group $E_H$ is generated by the set $A$.

\noindent
Let $w$ be an element of $E_H$. Then $\pi(w)\in H$ and $\pi(w)=y_1^{e_1}y_2^{e_2}\dots y_n^{e_n}$ for some $y_i\in Y$ and $e_i\in \mathbb{Z}$. For each $i$, we have $\phi (y_i)=y'_i\in E_H$ and $\pi (w)=\pi (y'_1)^{e_1}\pi(y_2')^{e_2}\dots \pi(y_n')^{e_n}=\pi((y_1')^{e_1}(y_2')^{e_2}\dots (y_n')^{e_n})$. Then we have $w((y_1')^{e_1}(y_2')^{e_2}\dots (y_n')^{e_n})^{-1} \in$ ker($\pi)=\pi_1(S)$. Therefore, $w((y_1')^{e_1}(y_2')^{e_2}\dots (y_n')^{e_n})^{-1}=_{E_H} x_1^{f_1}x_2^{f_2}\dots x_m^{f_m}$ for some $x_j\in X$ and $f_i\in\mathbb{Z}$.
Hence, $w=_{E_H} x_1^{f_1}x_2^{f_2}\dots x_m^{f_m}(y_1')^{e_1}(y_2')^{e_2}\dots (y_n')^{e_n}$ and therefore Claim 1 holds since $w$ is arbitrary.

\item Claim 2 :  The set $U$ is a set of defining relations for $E_H$ on $A$, that is, $E_H$ has the presentation $E_H=\langle A\,|\,U\rangle$.

\noindent
By construction, every element of $U$ is a relation in $E_H$. Suppose now $w\in F(A)$ is such that $w=1$ in $E_H$. 
Using the conjugation relations from $U$ and pushing the letters from $Y$ to the right, we can rewrite $w$ in the form $w=_{E_H} u (y_1')^{e_1}\dots (y_n')^{e_n}$ for some word $u\in F(X)\, , y_i\in Y,$ and $e_i\in\mathbb{Z}$. Since $\pi(w)=1$ in $H$, we have $y_1^{e_1}\dots y_n^{e_n}=1$ in $H$.  Then we can reduce the word $y_1^{e_1}\dots y_n^{e_n}$ to the empty words in $H$ using the relations from $Z$. Each application of a relation $z$ from $Z$ consists of replacing a subword of this relation by its complementary portion in $z$. Starting from the word $(y_1')^{e_1}\dots (y_n')^{e_n}$, for each such move we perform the corresponding move in $F(A)$ and replace a subword of $z'$ by the complementary portion of $z'v_z^{-1}$ there, and then pushing all newly created letters from $X$ to the left using the conjugation relations from $U$. Iterating this process, using the relations from $U$ we can rewrite $w$ as $w=_{E_H} u v$ where $ v\in F(X)$. Then $uv=_{\pi_1(S)} 1$, and we can rewrite $uv$ to the empty word  using the relations from $R$. Therefore, $w\in  \langle\langle U\rangle\rangle_{F(A)}$, as claimed.
\end{enumerate}

It follows that $\langle A\,|\,U\rangle $ is indeed a finite presentation for $E_H$. \end{proof}

For a closed surface $S$, possibly with a finite set of punctures, Mosher \cite{M95} showed that the mapping class group Mod($S$) of $S$ is automatic.
For an automatic group $G$ with an automatic structure $(L,A)$ where $L$ is a regular language in the free monoid $A^*$ on a finite set $A$ of semigroup generators for $G$, a subgroup $H$ of $G$ is called  \textit{L-rational} if its full preimage in $L$, $L\cap \pi^{-1}(H)$ where $\pi : A^* \to G$ is the natural monoid homomophism, is a regular language. See \cite{ECHMPT92} for further details.
Gersten and Short \cite{GS91} proved that a rational subgroup $H$ of an automatic group $G$ is finitely presented.
The algorithm given by Kapovich \cite{K96}, applied to such $H$, produces both a rationality constant for $H$ in $G$ and a finite presentation for $H$. 
Gersten and Short \cite{GS91} show that the subgroup $H$ is $L$-rational if and only if $H$ is $L$-quasiconvex in $G$. Since Morseness implies quasiconveixty, a Morse subgroup $H$ of an automatic group $G$ is rational, and therefore the algorithm in \cite{K96} applied to $H$ eventually terminates.

\begin{proof}[Second proof of Theorem \ref{main1}(\ref{main1-1}) when $S$ is a closed hyperbolic surface]
For a closed hyperbolic surface $S$, let $\langle X | R\rangle$ be the standard presentation for $\pi_1(S)$. Suppose that $H$ is a finitely generated subgroup of Mod($S$).
We run the \cite{K96} procedure  for detecting rationality of $H$ in Mod($S$). If it terminates, take a presentation for $H$ obtained from this procedure, compute a presentation for $E_H $ by using Lemma \ref{presentation} and then run Papasoglu's algorithm \cite{P95} for detecting hyperbolicity on $E_H$. If papasoglu's algorithm discovers that $E_H$ is hyperbolic, terminate the entire algorithm and declare that $H$ is stable in Mod($S$). 

It is known that $E_H$ is hyperbolic if and only if $H$ is convex cocompact \cite{FM02, H05}. Hence, by Theorem \ref{convex}, the above algorithm terminates if and only if $H$ is stable in Mod($S$). \end{proof}


\section{Right-angled Artin groups}\label{3}

For the material in this section, see \cite{C07, BC12} as background references. 

\begin{definition}[Right-angled Artin groups]
Let $\Gamma$ be a finite simplicial graph with vertex set $V(\Gamma)$ and edge set $E(\Gamma)\subset V\times V$. The \textit{right-angled Artin group} on $\Gamma$ has the finite group presentation:
\[
A_\Gamma:=\langle\,\,V(\Gamma)\,\, |\, \,v_iv_j=v_jv_i \text{ whenever } (v_i, v_j)\in E(\Gamma)\,\,\rangle .
\]
\end{definition}

For an induced subgraph $\Lambda$ of $\Gamma$ it is known that the subgroup of $A_\Gamma$ generated by $V(\Lambda)$ is the right-angled Artin group on $\Lambda$, and we denote this subgroup $A_\Lambda\le A_\Gamma$.

\begin{definition}
For two graphs $\Gamma_1$ and $\Gamma_2$, the \textit{join} of $\Gamma_1$ and $\Gamma_2$ is a graph obtained by connecting every vertex of $\Gamma_1$ to every vertex of $\Gamma_2$ by an edge.
If $|V(\Gamma_1)|=1$, $|V(\Gamma_2)|=1$ then the join of $\Gamma_1$ and $\Gamma_2$ is called a \textit{trivial join}.
A graph $\Gamma$ is \textit{anti-connected} if $\Gamma$ does not decompose as a nontrivial join.
For an induced subgraph $\Lambda$ of $\Gamma$ which decomposes as a nontrivial join, the subgroup $A_\Lambda$ is called a \textit{join subgroup} of $A_\Gamma$.
\end{definition}

\begin{definition}
Let $\Gamma$ be a finite simplicial, connected, and anti-connected graph.
An element $g$ of $A_\Gamma$ is called \textit{loxodromic} if the element $g$ is  not conjugate into a join subgroup of $A_\Gamma$ and  \textit{elliptic} otherwise. 
\end{definition}

\begin{definition}

Let $\Gamma$ be a finite simplicial, connected, and anti-connected graph.
A word $w$ over the alphabet $V(\Gamma)^{\pm 1}$ is said to be in a \textit{normal form} for $A_\Gamma$ if $w$ is freely reduced and does not contain any subwords of the form $x^{\epsilon} u x^{-\epsilon}$, where $x\in V(\Gamma), \epsilon=\pm 1$, and where $x$ commutes with every letter from $u$. A word $w$ over the alphabet $V(\Gamma)^{\pm 1}$ is said to be in a \textit{cyclically reduced normal form} for $A_\Gamma$ if $w$ is in a normal form for $A_\Gamma$ and every cyclic permutation of $w$ is in a normal form.
\end{definition}

For a finite connected graph  $\Gamma$, Kim and Koberda \cite{KK13, KK14} introduced the \textit{extension graph} $\Gamma^e$ of $\Gamma$.
The vertex set of $\Gamma^e$ is $\{v^g =g^{-1}vg \,|\, v \in V(\Gamma), \,g\in A_\Gamma\}$ and two distinct vertices $u^g$ and $v^h$ are adjacent if and only if they commute in $A_\Gamma$.
For a finite and connected graph $\Gamma$,  the extension graph $\Gamma^e$ is a quasi-tree and thus a hyperbolic metric space \cite{KK13}. We have the following characterizations of  loxodromic elements in $A_\Gamma$, see \cite{S89, BC12, KK13, KK14}.

\begin{theorem}\label{KMT17-0}
Let $\Gamma$ be a finite simplicial, connected, and anti-connected graph with at least two vertices, and let $g$ be an element of $ A_\Gamma$. Then the following are equivalent:
\begin{enumerate}[{(1)}]
\item The element $g$ is loxodromic.
\item The element $g$ acts as a rank 1 isometry on the universal cover $\widetilde{S_\Gamma}$ of the Salvetti complex $S_\Gamma$.
\item The centralizer $C_{A_\Gamma}(g)$ of $g$ is infinite cyclic.
\item The element $g$ acts as a loxodromic isometry of the extension graph $\Gamma^e$.
\item Some (equivalently, any) cyclically reduced normal form $g$ is not in a join subgroup of $A_\Gamma$.
\end{enumerate}

\end{theorem}

Note that Theorem \ref{KMT17-0}(5) provides an algorithm that, given a word in the generators of $A_\Gamma$, decides whether or not this word represents a loxodromic element. Koberda, Mangahas, and Taylor \cite{KMT17}  gave a characterization of stable subgroups of a right-angled Artin group as follows.

\begin{theorem}[Theorem 1.1 in \cite{KMT17}]\label{KMT17}
Let $\Gamma$ be a finite simplicial, connected, and anti-connected graph and let $H$ be a finitely generated subgroup of $A_\Gamma$. Then the following are equivalent:
\begin{enumerate}[{(1)}]
\item Some (any) orbit map from $H$ into $\Gamma^e$ is a quasi-isometric embedding.
\item The subgroup $H$ is stable in $A_\Gamma$. 
\item The subgroup $H$ is purely loxodromic, i.e., every nontrivial element of $H$ is loxodromic in $A_\Gamma$.
\end{enumerate}
\end{theorem}
 
Tran \cite{T17} and Genevois \cite{G17} showed, independently, that stability and Morseness are equivalent for an infinite index subgroup of the right-angled Artin group $A_\Gamma$.

\begin{theorem}[Theorem 1.16 in \cite{T17} or Theorem B.1 in \cite{G17}] \label{rightangled}
Let $\Gamma$ be a finite simplicial, connected, and anti-connected graph. Let $H$ be a finitely generated infinite index subgroup of $A_\Gamma$. Then $H$ is stable in $A_\Gamma$ if and only if $H$ is Morse in $A_\Gamma$.
\end{theorem}


In Section \ref{star} and Section \ref{complex},  we prove Theorem \ref{main2}(\ref{main2-1}) by two different methods.

\subsection{Star Length}\label{star}

Recall that for a graph $\Gamma$, the \textit{link} of a vertex $v$ in $\Gamma$, denoted by Lk$(v)$, is the set of the vertices in $\Gamma$ which are adjacent to $v$ and the \textit{star} of $v$, denoted by St$(v)$, is the union of Lk$(v$) and $\{v\}$. 
Kim and Koberda \cite{KK14} defined the \textit{star metric} on $A_\Gamma$ and showed that the extension graph $\Gamma^e$ with standard graph metric is quasi-isometric to the Cayley graph of $A_\Gamma$ with the star metric.

\begin{definition}[Star length]
For a right-angled Artin group $A_\Gamma$ and an element $g\in A_\Gamma$ the \textit{star length} of an element $g$ is the minimum $l$ such that $g$ can be written as the product of $l$ elements in $ \bigcup_{v\in V( \Gamma)} \langle \text{St}(v)\rangle$. We denote the star length $g$ by $|g|_*$. The star length induces a metric $d_*$ on $A_\Gamma$ by left invariance: $d_*(g,h):=|g^{-1}h|_*$.
\end{definition}

\begin{theorem}[Theorem 15  in \cite{KK14}]\label{KK14-1}
Let $\Gamma$ be a finite connected graph.
The metric spaces $(A_\Gamma, d_*)$ and $(\Gamma^e, d_{\Gamma^e})$ are quasi–-isometric, and the orbit map $A_\Gamma \to \Gamma^e,$ $g\mapsto v^g=g^{-1}vg$ (where $v$ is a base-vertex of $\Gamma^e$) is a quasi-isometry.
\end{theorem}

\begin{definition}
Let $\Gamma$ be a simplicial, finite, connected, and anti-connected graph.
A product $g= g_1\dots g_k$ is called a \textit{star-geodesic} if $g_i\in$ St$(v_i)$ for $i=1,\dots , k$ where $v_i\in V(\Gamma)$ and $k=|g|_*$.
A word $w$ over $V(\Gamma)^{\pm 1}$ is called a \textit{star-block} if there is some $v\in V(\Gamma)$ such that every letter of $w$ belongs to St($v$).

\end{definition}

The following lemma allows us to compute the star length for an element in $A_\Gamma$ by using its normal form. Note that this fact can also be derived from Lemma 20 in \cite{KK14}.

\begin{lemma}\label{star-length}
Let $\Gamma$ be a finite simplicial, connected, and anti-connected graph. For a nontrivial element $g\in A_\Gamma$ the star length $|g|_*$  is equal to the smallest $k$ such that there exists a word $w=w_1\dots w_k$ representing $g$ such that each $w_i$ is a nontrivial star-block and that $w$ is in a normal form for $A_\Gamma$. 
\end{lemma}

\begin{proof}
Take an element $1\ne g\in A_\Gamma$, and put $k=|g|_*$. Now look at all representations of $g$ as $g=w_1\dots w_k$ where each $w_i $ is a star-block word and among them choose a representation $w=w_1\dots w_k$ with $|w|= \sum_{i=1}^k |w_i|$ minimal possible. Note that this choice of $w$ implies that for all $1\le i<j\le k$, $|w_i\dots w_j|_*=j-i+1$.

We claim that this representation $g=w_1\dots w_k$ is in a normal form for $A_\Gamma$.
Note that the minimality assumption on $|w|$ implies that $w$ is a freely reduced word. 
If the word $w=w_1\dots w_k $ is in a normal form for $A_\Gamma$, we are done.  Now suppose that $w$ is not in a normal form for $A_\Gamma$.
Then $w$ contains a subword of the form $xux^{-1}$ where $x$ or $x^{-1}$ is a vertex of $\Gamma$ and $u$ is a nontrivial word and where $x$ commutes with every letter from $u$.
For simplicity, say $x\in V(\Gamma).$
Since all letters from $u$ commute with $x$, $u\in \langle \text{Star(}x)\rangle$. We take an innermost $xux^{-1}$ of this type so that $u$ is in normal form. 
Since we chose $w$ with $\sum_{i=1}^k |w_i|$ minimal possible, if $x$ occurs in $w_i$, then $x^{-1}$ comes from $w_j$ with $i<j$. 
There are several cases to consider:

\begin{enumerate}[{(i)}]
\item The letter $x$ occurs in $w_i$ and $x^{-1}$ occurs in $w_j$ where $j\ge i+3$.

\noindent
In this case, $w_{i+1}$ and $w_{i+2}$ are both subwords of $u$ in $xux^{-1}$. Since 
$u$ commutes with $x$ letterwise, the words $w_{i+1} $ and $w_{i+2}$ can be viewed as words in the generators of Star($x$). Then $g$ can be expressed as a product of strictly less than $k$ star-block, which contradicts $|g|_*=k.$ Hence, this case does not happen.

\noindent
\item  The letter $x$ occurs in $w_i$ and $x^{-1}$ occurs in $w_{i+2}$.

\noindent
The entire $w_{i+1}$ is a subword of $u$, and we can view $w_{i+1}$ as a word in the generators of Star($x$). We then rewrite $w_{i} w_{i+1} w_{i+2} =w_{i}'w_{i+1}'w_{i+2}'$, where $w_i=w_i'x$, $w_{i+1}'=u$, $w_{i+2}=x^{-1}w_{i+2}'$. We thus get a star-decomposition of $g$ as 
$w'=w_1\dots w_i'w_{i+1}'w_{i+2}'\dots w_k$ with a smaller $|w'|$ than $|w|$, which contradicts the choice of $w$.

\noindent
\item The letter $x$ occurs in $w_i$ and $x^{-1}$ occurs in $w_{i+1}$.

\noindent
In this case, $w_i=zxu_1,$ and $w_{i+1}=u_2x^{-1}y$ for some reduced words $u_1$ and $u_2$.
Since $u_1, u_2\in \langle \text{Star(}x)\rangle$,  $w_i=_{A_\Gamma} zu_1x$ and $w_{i+1}=_{A_\Gamma}x^{-1}u_2y$. Then $g$ is equal in $A_\Gamma$ to the word $w'=w_1\dots w_{i-1}(zu_1)(u_2y) w_{i+2}\dots w_k.$  
Note that $zu_1$ and $u_2y$ are star-block words. Hence, $|w'|<|w|$ which contradicts the minimal choice of $w$. 
\end{enumerate}
The claim now directly implies the statement of the lemma.\end{proof}

Lemma \ref{star-length} allows us to have the following algorithm.

\begin{corollary}\label{star-length2}
Let $\Gamma$ be a finite simplicial, connected, and anti-connected graph.
There exists an algorithm that, given a word $w$ in the generators of $A_\Gamma$ representing an element $g$, computes  $|g|_*$ and produces a star-geodesic representative of $g$. 
\end{corollary}

\begin{proof}
 Given a word $w$ in the generators of $A_\Gamma$, compute all normal forms of the element $g$ represented by $w$. Note that these normal forms can be computed algorithmically \cite{DKL12} and that they all are related to each other by a process of shuffling the letters using the commutativity relations in $A_\Gamma$ \cite{HM95}.  Among these normal forms, find the one which decomposes as a product of the smallest number of star-block words. Lemma \ref{star-length} implies that this numbers equals $|g|_*$. 
\end{proof}

\begin{proof}[First proof of Theorem \ref{main2}(\ref{main2-1})]
Pick a base-vertex $v\in V(\Gamma)$. We also regard $v$ as a base-vertex of the extension graph $\Gamma^e$.
Suppose that $H$ is a finitely generated subgroup of $A_\Gamma$ given by a finite generating set $S$ in $A_\Gamma$ such that $S=S^{-1}$. 
Let $f : A_\Gamma \to \Gamma^e$ be an orbit map sending $g$ to $v^g$ for every $g\in A_\Gamma$.
We can extend $f$ to the Cayely graph $\Gamma_{A_\Gamma}$ of $A_\Gamma$ sending an edge $(g, gs)$ to $(v^g, v^{gs})$. By Theorem \ref{KK14-1}, $f: (A_\Gamma, d_*)\to \Gamma^e$ is a quasi-isometry.
Let $g : (\Gamma_{A_\Gamma}, d) \to (\Gamma_{A_\Gamma}, d_*)$ be the identity map, where $d$ is the word metric on $\Gamma_{A_\Gamma}$.
Note that $g$ sends a geodesic of length $n$ in $(\Gamma_{A_\Gamma}, d)$ to 
a path of $d_*$-distance between its endpoints $\le n$ in $(\Gamma_{A_\Gamma}, d_*)$.
By Theorem \ref{KMT17}, the subgroup $H$ is stable in $A_\Gamma$ if and only if there exist integers $\lambda\geq 1, \epsilon \geq 0$ such that for every geodesic $w$ in $(\Gamma_{A_\Gamma}, d)$ the image $f\cdot g(w)$ is $(\lambda, \epsilon)$-quasigeodesic in $\Gamma^e$.
By Theorem \ref{KK14-1}, this condition is equivalent to the existence of $\lambda'\geq 1, \epsilon' \geq 0$ such that for every geodesic $w$ starting from 1 in $(\Gamma_{A_\Gamma}, d)$ the image $g(w)=w$ is $(\lambda', \epsilon')$-quasigeodesic in $(\Gamma_{A_\Gamma}, d_*)$.
The algorithm for deciding stability is as follows. 

Start enumerating all pairs of integers $\lambda\geq 1, \epsilon \geq 0$.
For every such pair $\lambda, \epsilon$, compute the constants $K$, $\lambda' \geq 1$ and $\epsilon'\geq 0$ from Lemma \ref{lem}.
Using the solution of the word problem in $A_\Gamma$, list all the geodesic edge-paths $w$ in $(\Gamma_{A_\Gamma}, d)$ of length at most $K$. 
For every such $w$,  check whether or not $w$  is a ($\lambda', \epsilon')$-quasigeodesic in $(\Gamma_{A_\Gamma}, d_*)$ by Lemma \ref{star-length2}. 
If the answer is `yes' for all such $w$, we terminate the algorithm and declare that $H$ is stable in $A_\Gamma$. Otherwise we proceed to the next pair of $\lambda$ and $\epsilon$. This algorithm terminates if and only if $H$ is stable. In parallel, we keep enumerating elements of $H$, find their cyclically reduced forms
and, using Theorem \ref{KMT17-0}  check whether the element of $H$ under consideration is loxodromic. 
If we find a non-loxodromic element of $H$, we terminate the algorithm and declare that $H$ is not stable by Theorem \ref{KMT17}. 

Taken together, this procedure provides a complete algorithm for detecting stability of $H$ in $A_\Gamma$.\end{proof}

\subsection {Cube complex }\label{complex}

In this section, we use \cite{H08, MT16} as background references.
Let $\Gamma$ be a finite simplicial graph.
Recall the construction of the Salvetti complex $S_\Gamma$ for $A_\Gamma$.
Begin with a wedge of circles attached to a vertex $x_0$ and labeled by the elements of $V(\Gamma)$. For each edge, say from $v_i$ to $v_j$ in $\Gamma$, attach a 2-torus with boundary labeled by the relator $v_i^{-1}v_j^{-1}v_iv_j$. For each triangle in $\Gamma$ connecting three vertices $v_i,v_j,v_k$, attach a 3-torus with faces corresponding to the tori for the three edges of the triangle. Continue this process, attaching a $k$-torus for each set of $k$ mutually commuting generators. The resulting space is called the \textit{Salvetti complex} $S_\Gamma$ for $A_\Gamma$.
Note that $\pi_1(S_\Gamma, x_0)=A_\Gamma$ and that the universal cover $\widetilde{S_\Gamma}$ of $S_\Gamma$ is a CAT(0) cube complex.

\begin{lemma}[Lemma 3.5 in \cite{MT16}]\label{MT16.3.5}
Let $\Gamma$ be a finite simplicial graph and let $H$ be a quasiconvex subgroup of $A_\Gamma$ with respect to the standard generators of $A_\Gamma$. Then there exists a pointed finite connected cube complex $(C, x)$ and a cubical local isometry $\phi : (C, x) \rightarrow (S_\Gamma, x_0)$ with $H =\phi_\ast(\pi_1(C, x))$. 
\end{lemma}

\begin{lemma}[Lemma 3.6 in \cite{MT16}]\label{MT16.3.6}
Let $\Gamma$ be a finite simplicial graph and let $H$  be a quasiconvex subgroup of $A_\Gamma$ with respect to the standard generators of $A_\Gamma$.
Let $C$ be the pointed finite connected cube complex as in Lemma \ref {MT16.3.5}.
Then oriented edge-loops at $x \in C$ which are combinatorial local geodesics are in bijective correspondence with minimal length words in $H$ with respect to the standard generators of $A_\Gamma$.

\end{lemma}

\begin{lemma}\label{MT16.3.7}
Let $\Gamma$ be a finite simplicial graph and let $H$ be a quasiconvex subgroup of $A_\Gamma$ with respect to the standard generators of $A_\Gamma$.
Let $(C, x)$ be the pointed finite connected cube complex as in Lemma \ref {MT16.3.5}. Then
\begin{enumerate}[{(i)}]
\item for an infinite order element $h\in H$, there exists a cyclically reduced normal form $w$ of $h$ such that $w$ is corresponding to a nontrivial closed loop based at some vertex (not necessarily equal to $x$) of $C$, and
\item the subgroup $H$ contains a non-loxodromic element if and only if there is a nontrivial simple loop in $C$ whose label is a join word.

\end{enumerate}
\end{lemma}

\begin{proof}
Let $h$ be an infinite order element in $H$.
By Lemma \ref{MT16.3.6}, a minimal length representative $w$ of $h$ corresponds to an edge-loop based at $x$ which is a combinatorial local geodesic in $C$.
If $w$ is cyclically reduced then we are done.
Suppose now that $w$ is not cyclically reduced. Then there exists a normal form $w'$ of $h$ which has the form $w'=u (v_{i}\dots v_j) u^{-1}$ where $u$ is nontrivial and $v_{i}\dots v_j$ is a cyclically reduced normal form of $h$, with $v_i\in V(\Gamma)^{\pm}$.
Note that, by Lemma \ref{MT16.3.5}, the map from the complex $C$ to the Salvetti complex $S_\Gamma$ is an immersion so that the 1-skeleton of $C$ is a folded graph in the free group sense, see \cite{KM02} for more details. That guarantees that there is at most one path starting at the base point $x$ and labeled by $u$. The end point of this path is a vertex $y$ of $C$ and the portion of $w$ corresponding to $v_{i}\dots v_j$ gives a loop in $C$ based at $y$ and labeled by $u_i\dots u_j$. This completes the proof of part (i). 

For part (ii), first suppose that there is a simple closed loop based at a vertex $y$ in $C$ whose label is a join word.  
Let $u$ be the label of the path starting from $x$ terminating at $y$ and let $v$ be the label of the simple closed loop. Since $\pi_1(C,x)=H$, the non-loxodromic element $uvu^{-1}$ is in $H$.
Suppose now that $H$ contains a nontrivial non-loxodromic element $h$. Then by part (i) there exists a nontrivial loop $\alpha$ in $C$ at some vertex $y$ with the label $w$ of $\alpha$ being a cyclically reduced normal form of $h$. Then $w$ is a join word, and every subword of $w$ is a join word. We can find a nontrivial subpath $\beta$ of $\alpha$ such that $\beta$ is a simple closed loop at some vertex of $C$. The label of $\beta$ is a subword of $w$ and thus is a join word. This completes the proof of part (ii).  \end{proof}

\begin{proof}[Seond proof of Theorem \ref{main2}(\ref{main2-1})]
Let $H$ be a finitely generated subgroup of $A_\Gamma$ given by a finite generating set $S$. 
Start enumerating candidate finite base-pointed connected cube complexes $(C,x)$ admitting cubical locally isometric maps to $S_\Gamma$.
For each such $(C,x)$, pick a finite generating set $Y$ for the fundamental group $\pi_1(C^{(1)},x)$ of the 1-skeleton of $C$. 
Start enumerating all words in $S^{\pm 1}$ and the labels of all loops in $C^{(1)}$ based at $x$. Then check, using the word problem for $A_\Gamma$, whether elements of $S$ all appear in the second list and whether all elements of $Y$ appear in the first list. 
Suppose that we find such $(C,x)$  with $\pi_1(C,x)=H$. Now check all finite simple loops in $C^{(1)}$. 
If all such loops are labeled by not a join word then we declare that $H$ is stable in $A_\Gamma$, and otherwise we declare that $H$ is not stable.
In parallel, with the above process involving cube complexes $(C,x)$, we keep enumerating elements of $H$ and check them for being non-loxodromic. If there is a non-loxodromic element in $H$, then we declare that $H$ is not stable. 

Note that if the subgroup $H$ is stable, then Lemma \ref{MT16.3.5} guarantees that we eventually find such $(C,x)$ where we check the existence of a non-loxodromic element of $H$ by Lemma \ref{MT16.3.7}(2).
Note also that by Theorem \ref{KMT17} the subgroup $H$ is stable if and only if $H$ is purely loxodromic. 
Therefore, this procedure gives a complete algorithm for deciding whether or not $H$ is stable in $A_\Gamma$.
\end{proof}

\begin{proof}[Proof of Theorem \ref{main2}(\ref{main2-2})]
Let $H$ be a finitely generated subgroup of $A_\Gamma$ given by a finite generating set $S$.
We first run one of the complete algorithms in Theorem \ref{main2}(\ref{main2-1}) for deciding stability of $H$ in $A_\Gamma$. If the algorithm terminates with the conclusion that $H$ is stable, then we declare that $H$ is Morse in $A_\Gamma$. Suppose now that the algorithm in Theorem \ref{main2}(\ref{main2-1}) terminates with the conclusion that $H$ is not stable in $A_\Gamma$. We then keep running the Todd-Coxter algorithm on $H$, for detecting finitenness of the index of $H$ in $A_\Gamma$. 
If the Todd-Coxeter algorithm terminates then we declare that $H$ is Morse in $A_\Gamma$. Otherwise we continue running the Todd-Coxeter algorithm forever.

Note that the subgroup $H$ is Morse in $A_\Gamma$ if and only if either $H$ is stable or $H$ has finite index in $A_\Gamma$ by Theorem \ref{rightangled}.
Hence, taken together, the above procedure provides a partial algorithm for detecting Morseness of $H$ in $A_\Gamma$. \end{proof}

\begin{remark}
Similar to mapping class groups, there is no known algorithm that, given a finitely generated subgroup $H$ of $A_\Gamma$, decides that $H$ has infinite index in $A_\Gamma$. If such an algorithm is found, we can improve Theorem \ref{main2}(\ref{main2-2})
to a complete algorithm deciding Morseness of $H$ in $A_\Gamma$.
\end{remark}




\section{Toral relatively hyperbolic groups}\label{4}

There are various definitions of a relatively hyperbolic group, see \cite{O06, H13} for more details.
In this paper we use the following definition of relative hyperbolicity, due to Bowditch \cite{B99}. 

\begin{definition}[Relatively hyperbolic groups]

Let $G$ be a finitely generated group and let $\mathbb P$ be a (possibly empty) finite collection of finitely generated subgroups of $G$.  
Suppose that $G$ acts on a $\delta$-hyperbolic graph $K$ with finite edge stabilizers and finitely many orbits of edges (and hence also of vertices). 
Suppose that each edge of $K$ is contained in only finitely many circuits of length $n$ for each integer $n$, and that $\mathbb{P}$ is a set of representatives of the conjugacy classes of infinite vertex stabilizers. Then $(G, \mathbb{P})$ is a \textit{relatively hyperbolic} group with respect to $\mathbb{P}$. An element $P$ of $\mathbb{P}$ is called a \textit{peripheral} subgroup of $G$.
\end{definition}

For a relatively hyperbolic group $(G, \mathbb{P}=\{P_1, \dots , P_n\})$, we allow the case $n=0$, in which situation the family $\mathbb{P}$ is empty and the group $G$ is word-hyperbolic.

\begin{definition}
Let $(G,\mathbb{P})$ be a relatively hyperbolic group. An element $g\in G$ is \textit{elliptic} if it has finite order, \textit{parabolic} if it has infinite order and is conjugate to a subgroup of some $P\in\mathbb{P}$, and \textit{hyperbolic} or \textit{loxodromic} otherwise.  A subgroup $H$ of $G$ is \textit{elliptic} if it is finite, \textit{parabolic} if it is infinite and contained in a conjugate of a peripheral subgroup $P\in \mathbb{P}$, and \textit{hyperbolic} otherwise.
\end{definition}

The notion of relatively quasiconvex subgroups plays an important role in the theory of relatively hyperbolic groups.
Note that there are several definitions of relative quasiconvexity, and they are equivalent for a finitely generated relatively hyperbolic group \cite{H13}.
Recall that for a group $G$ with a collection of subgroups $\mathbb{P} = \{P_1,\dots ,P_n\}$, a subset $S$ of $G$ is called \textit{relative generating set} for the pair $(G, \mathbb{P})$ if the set $S \cup P_1 \cup \dots \cup P_n$ generates $G$.

\begin{definition}[Relatively quasiconvex subgroups]\label{quasiconvex}
 Let $(G,\mathbb P)$ be a relatively hyperbolic group. 
A subgroup $H$ of $G$ is \textit{relatively quasiconvex} if the following holds. Let $S$ be some (any) finite relative generating set for $(G,\mathbb{P})$, and let $\mathcal{P}$ be the union of all $P \in \mathbb{P}$. Consider the Cayley graph $\Gamma =$ Cayley$(G, S \cup \mathcal{P})$ with all edges of length one. Let $d$ be some (any) proper, left invariant metric on $G$. Then there is a constant $k = k(S,d)$ such that for each geodesic $c$ in $\Gamma$ connecting two points of $H$, every vertex of $c$ lies within a $d-$distance $k$ of $H$. 
\end{definition}

It is known that every undistorted finitely generated subgroup of a relatively hyperbolic group $(G, \mathbb{P})$ is relatively quasiconvex in $G$ \cite{H13}. Hence, stable and Morse subgroups of $G$ are relatively quasiconvex.
Also, each peripheral subgroup $P\in \mathbb{P}$ is relatively quasiconvex in $G$ since $P\in \mathbb{P}$ is Morse in $G$ \cite{DS05}.
For an undistorted subgroup $H$ of $G$, Tran \cite{T17} gave the following complete characterizations of stability and Morseness of $H$.
Recall our convention regarding denoting conguates of elements and of subgroups for a group $G$: $x^g=g^{-1}xg$, $H^g=g^{-1}Hg$ for $x, g\in G$ and $H\leqslant G$.

\begin{theorem}[Theorem 1.9 in \cite{T17}]\label{T17}
Let $(G, \mathbb{P})$ be a finitely generated relatively hyperbolic group and let $H$ be an undistorted  finitely generated subgroup of $G$. Then the following are equivalent:
\begin{enumerate}[{(1)}]
\item The subgroup $H$ is Morse in $G$.
\item The subgroup $H \cap P^g$ is Morse in $P^g$ for each conjugate $P^g$ of a peripheral subgroup in $\mathbb{P}$.
\item The subgroup $H \cap P^g$ is Morse in $G$ for each conjugate $P^g$ of a peripheral subgroup in $\mathbb{P}$.
\end{enumerate}
\end{theorem}

\begin{corollary}[Corollary 1.10 in \cite{T17}]\label{T172}
Let $(G, \mathbb{P})$ be a finitely generated relatively hyperbolic group and let $H$ be an undistorted finitely generated subgroup of $G$. Then the following are equivalent:
\begin{enumerate}[{(1)}]
\item The subgroup $H$ is stable in $G$.
\item The subgroup $H \cap P^g$ is stable in $P^g$ for each conjugate $P^g$ of a peripheral subgroup in $\mathbb{P}$.
\item The subgroup $H \cap P^g$ is stable in $G$ for each conjugate $P^g$ of a peripheral subgroup in $\mathbb{P}$.
\end{enumerate}
\end{corollary}

We now concentrate on a particular type of a relatively hyperbolic group, namely a toral relatively hyerbolic group.
A relatively hyperbolic group $(G, \mathbb{P})$ is called a \textit{toral} if $G$ is torsion-free and the elements of $\mathbb{P}$ are finitely generated free abelian non-cyclic subgroups of $G$.


\begin{lemma}\label{undistorted}
Let $(G,\mathbb{P})$ be a relatively hyperbolic group
where every $P\in \mathbb{P}$ is finitely generated abelian. Then every relatively quasiconvex subgroup of $G$ is undistorted.
\end{lemma}

\begin{proof}
Let $H$ be a relatively quasiconvex subgroup of $G$.  It is known that the distortion of $H$ in $G$ is a combination of the distortions of the infinite subgroups $H^{g} \cap P$ of $P\in \mathbb{P}$, see Theorem 1.4 in \cite{H13}. Since a peripheral subgroup  $P\in \mathbb{P}$ is finite generated abelian, the distortion of any infinite subgroup of $P$ is linear. Hence, the distortion of $H$ in $G$ is linear, that is, $H$ is undistorted in $G$.
\end{proof}

\subsection{}
For the remainder of this section, except for Proposition \ref{loxodromicelement} and Corollary \ref{main5}, we assume that $(G,\mathbb{P})$ is a toral relatively hyperbolic group where the finite family $\mathbb{P}$ of free abelian non-cyclic groups is nonempty. Also note that in the case where $\mathbb{P}$ is empty, $G$ is torsion-free word hyperbolic. In this case, for a finitely generated subgroup $H$ of $G$, being Morse is equivalent to being stable, it is equivalent to being quasiconvex, and it is also equivalent to being undistorted. The conclusions of Theorem \ref{main3} then follow from Proposition \ref{K96}.

Note that for a finitely generated free abelian non-cyclic group $P$, the only stable subgroup of $P$ is trivial, and a Morse subgroup is either trivial or has finite index in $P$. Hence, Theorem \ref{T17} and Corollary \ref{T172} imply the following:

\begin{corollary}\label{corollary}
Let $(G, \mathbb{P})$ be a toral relatively hyperbolic group and let be $H$ an undistorted finitely generated subgroup of $G$.
\begin{enumerate}[{(i)}]
\item The subgroup $H$ is stable in $G$ if and only if $H \cap P^g$ is trivial for each conjugate $P^g$ of a peripheral subgroup in $\mathbb{P}$.
\item The subgroup $H$ is Morse in $G$ if and only if $H \cap P^g$ either is trivial or has finite index in $P^g$ for each conjugate $P^g$ of a peripheral subgroup in $\mathbb{P}$.
 
\end{enumerate}
\end{corollary}

Corollary \ref{corollary} says that we only need to check the intersection $H \cap P^g$ for each conjugate $P^g$ of peripheral subgroup in $\mathbb{P}$ to detect stability or Morseness of $H$ in $G$.
Kharlampovich, Miasnikov, and Weil \cite{KMW17} provided a partial algorithm for computing the intersection of two given relatively quasiconvex subgroups with ``peripherally finite index'' of a toral relatively hyperbolic group $(G,\mathbb{P})$.


\begin{definition}[Peripherally finite index]
A subgroup $H$ of a finitely generated relatively hyperbolic group $(G, \mathbb{P})$ has \textit{peripherally finite index} in $G$, if, for each peripheral subgroup $P\in \mathbb{P}$ and each $g \in G$, the subgroup $H^g \cap P$ is either finite or has finite index in $P$. 
\end{definition}

\begin{definition}
A subgroup $P$ of a group $G$ is called \textit{almost malnornal} if for every $g\in G\setminus P$ the intersection $P\cap P^g$ is finite.
A family $\{P_1, \dots , P_k\}$ of subgroups of $G$ is called almost malnormal if whenever $g\in G, P_i$, and $P_j$ are such that $P_i\cap P_j^g$ is infinite then $i=j$ and $g\in P_i$ (so that $P_i=P_j^g)$. 
\end{definition}

For a relatively hyperbolic group $(G,\mathbb{P})$ it is known that the family $\mathbb{P}$ is almost malnormal in $G$. 
Hence, every peripheral subgroup $P\in \mathbb{P}$ has peripherally finite index in $G$.

\begin{theorem}\label{KMW17-2}
 Let $(G, \mathbb{P})$ be relatively hyperbolic where every $P\in\mathbb{P}$ is finitely generated abelian. Let $H$ be a relatively quasiconvex subgroup of G. Then: 
 
 \begin{enumerate}[{(i)}]
\item\label{i}  Every infinite parabolic subgroup of $H$ is contained in a unique maximal parabolic subgroup of $H$.

\item \label{ii} There are only finitely many $H$-conjugacy classes of maximal infinite parabolic subgroups of $H$.
\end{enumerate}
\end{theorem}

Theorem \ref{KMW17-2} follows from Proposition 7.19 in \cite{KMW17} because every $P\in \mathbb{P}$ is relatively quasoconvex. 
Note that Proposition 7.19 in \cite{KMW17} assumes that $G$ is a toral relatively hyperbolic group but the proof also works for a relatively hyperbolic group where all peripheral subgroups are finitely generated ablelian.
Also note that in a relatively hyperbolic group $(G,\mathbb{P})$, every infinite parabolic subgroup $Q$ of $G$ is contained in a unique conjugate $P^g$ of some $P\in \mathbb{P}$.

\begin{definition}\label{definduced}
Let $(G, \mathbb{P})$ be a relatively hyperbolic group where all peripheral subgroups are finitely generated abelian.
Let $H$ be a relatively quasiconvex subgroup of $G$ and let $\mathbb{D}$ be a collection of representatives of $H$-conjugacy classes of maximal infinite parabolic subgroups of $H$. The collection $\mathbb{D}$ is called the \textit{induced peripheral structure} for $H$ from $(G,\mathbb{P})$.
\end{definition}

Note that if $(G, \mathbb{P}), (H,\mathbb{D})$ are as in Definition \ref{definduced}, and if $g\in G, P\in \mathbb{P}$ are such that $H\cap P^g$ is infinite then $H\cap P^g$ is conjugate in $H$ to some $D\in \mathbb{D}$. 
That is, $\mathbb{D}$ is the set of representatives of $H$-conjugacy classes of infinite subgroups of $H$ of the form $H\cap P^g$ where $P\in \mathbb{P}$ and $g\in G$.
Moreover, the collection $\mathbb{D}$ is finite and $(H,\mathbb{D})$ is relatively hyperbolic \cite{GM17}.
Note that the subgroup $H$ has peripherally finite index if and only if for every $D\in \mathbb{D}$, whenever $D\le P^g$, $P\in \mathbb{P}$ then $D$ has finite index in $P^g$.

\begin{remark}\label{induced}

Let $(G, \mathbb{P})$ be a relatively hyperbolic group where all peripheral subgroups are finitely generated abelian and let $H$ be a relatively quasiconvex subgroup of $G$ with induced peripheral structure $\mathbb{D}$. 
Groves and Manning \cite{GM17} gave a definition of relatively quasiconvexity of $H$ which is equivalent to Definition \ref{quasiconvex}, and their work implies that a peripheral structure $\mathbb{D}$ on $H$ compatible with $P$ is unique in the following sense (see Definition 2.9 and the following paragraph in \cite{GM17}).
Suppose that $\mathbb{D}'$ is a finite family of infinite subgroups $D'$ of $H$ such that each $D'$ is infinite parabolic in $G$ and such that $(H,\mathbb{D}')$ is relatively hyperbolic. Then there exists a bijective correspondence between families $\mathbb{D}$ and $\mathbb{D}'$ such that if $D$ is sent to $D'$ under this correspondence then for some $h\in H$ the subgroups $D^h=D'$.
 \end{remark}

\begin{theorem}\label{KMW17-3}
Let $(G,\mathbb{P})$  be a relatively hyperbolic group such that all $P\in \mathbb{P}$ are finitely generated abelian. 
Then there is a partial algorithm which, a given finite tuples generating subgroup $H$ of $G$,
\begin{itemize}
\item 
halts if and only if the subgroup $H$ is relatively quasiconvex with peripherally finite index;
\item when it halts, computes (by producing their finite generating sets) a family $\mathbb{D}$ as in Definition \ref{definduced}.
\end{itemize}

\end{theorem}

Theorem \ref{KMW17-3} follows from Proposition 7.20 in \cite{KMW17} because every peripheral subgroup $P\in \mathbb{P}$ is relatively quasiconvex with peripherally finite index.
Note that Proposition 7.20 in \cite{KMW17} assumes that $G$ is a toral relatively hyperbolic group but their proof also works for a relatively hyperbolic group where all peripheral subgroups are finitely generated ablelian.
Moreover, Corollary 7.9 in \cite{KMW17} implies the following proposition that allows us, in particular, to find a finite generating set for an element of such $\mathbb{D}$ as in Theorem \ref{KMW17-3}.

\begin{proposition}\label{4.14}
Let $(G,\mathbb{P})$  be a relatively hyperbolic group such that all $P\in \mathbb{P}$ are finitely generated abelian. 
Then there is a partial algorithm that, given finite generating sets for relatively quasiconvex subgroups $H$, $K$ that have peripherally finite index in $G$, computes a finite generating set for $H\cap K$. 
\end{proposition}

\begin{proof}[Proof of Theorem \ref{main3}(\ref{main3-1})]
Let $H$ be a finitely generated subgroup of a toral relatively hyperbolic group $(G, \mathbb{P})$ given by a finite generating set of $H$.
We run the partial algorithm in Theorem \ref{KMW17-3} on the subgroup $H$. 
Suppose that the partial algorithm terminates, determines that $H$ is relatively quasiconvex of peripherally finite index in $(G, \mathbb{P})$, and computes such a family $\mathbb{D}$ as in Definition \ref{definduced}. If the family $\mathbb{D}$ is empty then we declare that $H$ is stable in $G$.

Note that if the subgroup $H$ is stable in $G$, then $H$ is relatively quasiconvex and has peripherally finite index in $G$ by Corollary \ref{corollary}. 
Therefore, if the subgroup $H$ is stable, the partial algorithm in Theorem \ref{KMW17-3} for $H$ eventually terminates. 
Conversely, by Lemma \ref{undistorted} and Corollary \ref{corollary}, if this algorithm terminates then $H$ is stable in $G$.  Thus the above procedure does detect stability of $H$ in $G$ as, required.\end{proof}

\begin{proof}[Proof of  Theorem \ref{main3}(\ref{main3-2})]
Let $H$ be a undisorted subgroup of a toral relatively hyperbolic group $(G, \mathbb{P})$. 
We run the partial algorithm in Theorem \ref{main3}(\ref{main3-1}).
If the algorithm halts then we declare that $H$ is stable in $G$.
In parallel, we enumerate elements of $H$ in $G$, enumerate conjugates of elements of $P$ for each  $P\in \mathbb{P}$, and look for an infinite order element in some $H\cap P^g$. If we find an infinite order element, then we declare that $H$ is not stable.
By Corollary \ref{corollary}, this procedure decides whether or not $H$ is stable in $G$.
\end{proof}

\begin{proof}[Proof of Theorem \ref{main3}(\ref{main3-3})]
Let $H$ be a finitely generated subgroup of a toral relatively hyperbolic group $G$. 
We run the partial algorithm in Theorem \ref{KMW17-3} on $H$.
Suppose that the partial algorithm terminates, determines that $H$ is relatively quasiconvex and peripherally of finite index in $(G,\mathbb{P})$, and computes such a family $\mathbb{D}$ as in Definition \ref{definduced}. If the family $\mathbb{D}$ are empty we declare that $H$ is Morse.
Otherwise, for each infinite subgroup $U=H\cap P^g$ in the collection $\mathbb{D}$, compute the index of $U$ in the finitely generated free abelian group $P^g$.
If all such subgroups $U$ have finite index in the corresponding $P^g$, we declare that $H$ is Morse in $G$ and terminate the procedure. 

Note that if the subgroup $H$ is Morse in $G$, then $H$ is relatively quasiconvex, and has peripherally finite index by Corollary \ref{corollary}.  
Therefore, if $H$ is Morse, the partial algorithm in Theorem \ref{KMW17-3} for $H$ eventually terminates.
Conversely, by Lemma \ref{undistorted} and Corollary \ref{corollary}, if the above procedure terminates then $H$ is Morse in $G$. Thus, this partial algorithm detects Morseness of $H$ in $G$, as required.
\end{proof}

\subsection{}
To improve Theorem \ref{main3}(\ref{main3-3}) to a complete algorithm deciding whether or not an undistorted subgroup $H$ is Morse, that is,  to prove Theorem \ref{main3}(\ref{main3-4}), 
we need to be able to decide whether or not $H$ has peripherally finite index in $G$.
For solving this problem,  we use the algorithms given by Kharlampovich, Miasnikov, and Weil \cite{KMW17} combined with the Groves and Manning's result \cite{GM17} on relatively hyperbolic Dehn fillings. Before stating Groves and Manning's result on relatively hyperbolic Dehn fillings, we recall some definitions, see \cite{O07, GM08, GM17} for further details.

\begin{definition}[Dehn fillings]
Let $(G, \mathbb{P})$ be a relatively hyperbolic group. A \textit{Dehn filling} of $(G, \mathbb{P})$ is the quotient $G/\langle \langle \bigcup N_P \rangle \rangle$ determined by normal subgroups $N_P\trianglelefteq P\in \mathbb{P}$, together with the quotient map $\pi : G\to \bar{G}$.
We denote the quotient $G/\langle \langle \bigcup N_P\rangle \rangle$ by $\bar{G}$.
\end{definition}

\begin{definition}
Let $G$ be a relatively hyperbolic group and let $\pi : G \to \bar{G}$ be a Dehn filling of $G$ with $\bar{G}=G/\langle \langle \bigcup N_P\rangle \rangle$.
 For a finite subset $Z \subset \bigcup_{P\in \mathbb{P}} P\setminus \{1\}$, we say that a Dehn filling is \textit{$Z$-long} if $N_P\cap Z=\emptyset$ for all $P\in \mathbb{P}$. We say that a statement holds for \textit{all sufficiently long fillings} if there exists a finite set $Z$ such that the statement holds for all $Z$-long fillings.
\end{definition}

Osin \cite{O07} and Groves and Manning \cite{GM08} proved, independently, that for a relatively hyperbolic group $G$, there exists a finite set $Z=Z(G)$ such that any $Z$-long Dehn filling $\bar{G}$ with $N_P \cap Z=\emptyset$ for all $P\in \mathbb{P}$ is again a relatively hyperbolic group:

\begin{theorem}\cite{O07, GM08}\label{OGM}
Let $(G, \mathbb{P})$ be a relatively hyperbolic group. 
There exists a finite $Z \subset (\bigcup_{P\in \mathbb{P}} P)\setminus \{1\}$ such that for every $Z$-long filling $\pi: G \to \bar{G}=G/ \langle\langle \bigcup N_p\rangle\rangle$ we have
\begin{enumerate}
\item  for each $N_p\trianglelefteq  P$, the Dehn filling $\pi$ induces an embedding of $P/N_p$ in $\bar{G}$ whose image we identify with $P/N_p$,
\item $(\bar{G}, \{\,P/N_p\,|\, P\in \mathbb{P} \,\})$ is relatively hyperbolic,
\end{enumerate}

\end{theorem}


For a relatively hyperbolic group $G$ and a relatively quasiconvex subgroup $H$ of $G$, Groves and Manning \cite{GM17} proposed the notion of $H$-wide fillings, and studied the behavior of $H$ under sufficiently long and $H$-wide fillings.

\begin{definition}\label{def1}
Let $P$ be a group, let $B$ be a subgroup of $P$, and let $S$ be a finite set. A normal subgroup $N$ of $P$ is $(B, S)$-\textit{wide} in $P$ if whenever  $b \in B$, $s \in S$ are such that $bs \in N$, then $s \in B$.
\end{definition}

\begin{definition}\label{def2}
Let $(G, \mathbb{P})$ be relatively hyperbolic and let $(H, \mathbb{D})$ be a relatively quasiconvex subgroup with the induced peripheral structure $\mathbb{D}$ as in Definition \ref{definduced}. Then for any $D\in \mathbb{D}$ there exists $P_D\in \mathbb{P}$ and $c_D\in G$ so that $D\le P_D^{c_D}$.
Let $S \subset (\bigcup_{P\in \mathbb{P}} P)\setminus \{1\}$. A filling $\pi : G\to \bar{G}$ of $G$ is \textit{$(H, S)$-wide} if for any $D\in \mathbb{D}$ the normal subgroup $N_{P_D}$ is $(D^{{c_D}^{-1}},\, S\cap P_D)$-wide in $P_D$.
\end{definition}

Note that is $S\subseteq S'$ and a filling is $(H,S')$-wide, then this filling is also $(H,S)$-wide.

\begin{definition}
Let $(G, \mathbb{P})$ be relatively hyperbolic and let $H$ be a relatively quasiconvex subgroup.
We say that a property holds for \textit{all sufficiently long and H-wide fillings} if there is a finite set $S \subset  (\bigcup_{P\in \mathbb{P}}P) \setminus \{1\}$ so that the property holds for any $(H,S)-$wide filling $G \to G/\langle \langle \bigcup N_P\rangle \rangle $ where $N_P \cap S = \emptyset$ for each $N_P\trianglelefteq P$.
\end{definition}


Groves and Manning \cite{GM17} proved the following properties on the images of $H$ under sufficiently long and $H$-wide fillings.

\begin{proposition}[Proposition 4.5 in \cite{GM17}]\label{GM17-1}
Let $(G, \mathbb{P})$ be relatively hyperbolic and let $H$ be a relatively quasiconvex subgroup of  $G$. Then all sufficiently long and $H$-wide fillings $\pi : G \rightarrow \bar{G}$ the subgroup $\pi(H)$ is relatively quasiconvex in $\bar{G}$.
\end{proposition}

The following theorem is a special case of Proposition 6.2 in \cite{GM17} which shows that if all peripheral subgroups of $(G, \mathbb{P})$ are finitely generated free abelian and a subgroup $H$ is relatively quasiconvex in $G$, then we can find sufficiently long and $H$-wide fillings. Specifically, Theorem \ref{GM17} is obtained by applying Proposition 6.2 of \cite{GM17} to the family $\mathbb{H}=\{H\} \cup \mathbb{P}$, where $H$ is a relatively quasiconvex subgroup of $(G,\mathbb{P})$. Note that in this case, as follows from Definition \ref{def1} and Definition \ref{def2}, a filling $\pi$ of $(G,\mathbb{P})$ is $(H,S)$-wide if and only if $\pi$ is $(H',S)$-wide for every $H'\in \mathbb{H}$.

\begin{theorem}\label{GM17}
Suppose that $(G, \mathbb{P})$ is relatively hyperbolic, that each element of $\mathbb{P}$ is finitely generated free abelian, that $H$ is a relatively quasiconvex subgroup of $G$, and that $S\subset( \bigcup_{P\in \mathbb{P}} P)\setminus \{1\}$  is a finite set. Then there exist finite index subgroups $\{ K_p \trianglelefteq P \,|\, P\in \mathbb{P} \}$ so that for any subgroups $N_P\le K_P$, the filling

\begin{center}
 $G\to G/\langle\langle \bigcup N_P \rangle \rangle\rangle=\bar{G}$ 
\end{center}
\noindent
is ($H, S)$–-wide.
Moreover, for an element $b\in G$ and $P\in \mathbb{P}$, if $1\not\in  PHb$ then there is no element of $\langle\langle \bigcup N_P \rangle \rangle\rangle$ in $PHb$, that is, $1\not\in \pi(PHb)=\pi(P)\pi(H)\pi(b)$.
\end{theorem}


\begin{definition}
Let $(G, \mathbb{P})$ be a relatively hyperbolic group where all peripheral subgroups are finitely generated abelian and let $(H, \mathbb{D})$ be a relatively quasiconvex subgroup of $G$ with the induced peripheral structure $\mathbb{D}$ as in Definition \ref{definduced}.
For every $D\in \mathbb{D}$ there exists $P_D\in \mathbb{P}$ and $c_D\in G$ so that $D\le P_D^{c_D}$.
For a Dehn filling $\pi :G\to \bar{G}=G/\langle \langle \bigcup N_P\rangle \rangle$ the \textit{induced filling kernels} for $(H, \mathbb{D})$ are the collection $\mathcal{D}=\{D\cap N_{P_D}^{c_D} \,| \, D\in \mathbb{D}\}$. These defines the \textit{induced filling} $\pi' : H\to \bar{H}=H/\langle\langle \mathcal{D}\rangle\rangle$ of $H$.
We denote $\bar{D} =\pi'(D)$ for $D\in \mathbb{D}$ and denote by $\bar{ \mathbb{D}}$  the list of all those $\bar{ D}$, where $D\in \mathbb{D}$, such that $\bar{D}$ is infinite.
\end{definition}

\begin{proposition}[Proposition 4.6 in \cite{GM17}]\label{GM17-2}
Let $(G, \mathbb{P})$ be relatively hyperbolic and let $H$ be a relatively quasiconvex subgroup of  $G$. For sufficiently long and $H$-wide fillings $\pi : G \to \bar{G}$ the map from the induced filling $\pi(H)$ of $H$ to $\bar{G}$ is injective.
\end{proposition}


We say that two finite lists $\mathbb{A}= A_1, \dots, A_k$ and $\mathbb{B}= B_1,\dots, B_s$ of infinite subgroups of a group $W$ are \textit{the same up to conjugation} in $W$, if  $k=s$ and there exists a permutation $\sigma\in S_k$ such that for every $1\le i\le k$, $B_{\sigma(i)}=A_i^{w_i}$ for some $w_i\in W$.

\begin{proposition}\label{image}
Let $(G, \mathbb{P})$ be relatively hyperbolic and let $H$ be a relatively quasiconvex subgroup of $G$ with the induced peripheral structure $\mathbb{D}$ from $G$. For sufficiently long and $H$-wide fillings, the induced peripheral structure on $\pi(H)$ from $\bar{G}$ is the same as the peripheral structure $\bar {\mathbb{D}}$ on $\bar{H}$, up to conjugation in $\bar H=\pi(H)$.
\end{proposition}

 \begin{proof}
By Theorem \ref{OGM} and Proposition \ref{GM17-1}, for a sufficiently long and $H$-wide filling $\pi :G\to \bar{G}$, the image $\pi(H)$ is relatively quasiconvex in the new relatively hyperbolic group $(\bar{G}, \{\pi(P)\,|\,P\in \mathbb{P}\})$. 
The induced peripheral structure on $\pi(H)$ from $\bar{G}$ consists of the infinite intersections $\pi(H)\cap{\pi(P)}^h$ where $h\in\bar{G}$.
Remark \ref{induced} and Proposition \ref{GM17-2} imply that if for $D\in \mathbb{D}$ the image 
$\pi(D)$ is infinite then $\pi(D)=\pi(H)\cap{\pi(P)}^h$ for some $h \in \bar{G}$.
Thus, the conclusion of Proposition \ref{image} holds as required.
 \end{proof}


\begin{proposition}\label{finite}
Let $(G,\mathbb{P})$ be a toral relatively hyperbolic group with $\mathbb{P}=\{P_1, \dots, P_k\}$.  Let $H$ be a relatively quasiconvex subgroup of $G$.
Let $a\in G$ be a fixed element of $G$. Then for all fillings $\pi :G \to \bar{G} $ that are sufficiently long and $H$-wide the following holds:

If  $\pi(H^a\cap P)$ is finite where $P\in \mathbb{P}$ then $ \pi(H)^{\pi(a)}\cap \pi(P)$ is also finite.
\end{proposition}

\begin{proof}

By replacing $H$ by $H^a$, without loss of generality, we can assume that $a=1$. Suppose that $\pi(H\cap P)$ is finite but $\pi(H)\cap \pi(P)$ is infinite. Take an element $\pi(u)$ of infinite order in $\pi(H)\cap \pi(P)$ where $u\in P$. Since the element $\pi(u)$ is parabolic of infinite order in $\pi(H)$, Remark \ref{induced} implies that $\pi(u)$ have the form $\pi(v^{bh})=\pi(h^{-1} b^{-1} v b h)$ where $h\in H$ and where $v\in H^{b^{-1}}\cap P$, and where $H\cap P^b$ is another element of $\mathbb{D}$, different from $H\cap P$. Hence, $\pi(u)=\pi( v^{bh})=\pi(v)^{\pi(bh)}$.
The groups $H\cap P$ and $H\cap P^{b}$ from $\mathbb{D}$ are in different $H$-conjugacy classes and therefore $b\not\in  PH$ (since $P$ is abelian).
In $\bar G=\pi(G)$, the elements $\pi(u)$ and $\pi(v)$ are infinite order elements of the finitely generated infinite abelian group $\pi(P)$ such that $\pi(u)$ and $\pi(v)$ are conjugate in $\bar{G}$. This implies that $\pi(u)=\pi(v)$ by almost-malnormality of $\pi(P)$ in $\bar{G}$. Then we have $\pi(v)=\pi(u)=\pi(v)^{\pi(bh)}$. Since $\pi(v)$ has infinite order in $\pi(P)$, almost malnormality of $\pi(P)$ implies that $ \pi(bh)\in \pi(P)$ and hence $\pi(b)\in \pi(P)\pi(H)$. However, the fact that $b\not\in PH$ implies, by the second part of  Proposition \ref{GM17}, $\pi(b)\not\in \pi(H)\pi(P)$. This gives a contradiction. 
\end{proof}

We now define a particular Dehn filling of $G$ having some conditions on $H$, namely a benign Dehn filling with respect to $H$.

\begin{definition}[Benign Dehn fillings] \label{benign}
Let $(G,\mathbb{P})$ be a relatively hyperbolic group where all peripheral subgroups of $G$ are finitely generated free abelian, and where $\mathbb{P}=\{P_1, \dots, P_k\}.$ Let $H$ be a relatively quasiconvex subgroup of $G$.
Let $Z\subset( \bigcup_{P\in \mathbb{P}} P)\setminus \{1\}$ be a finite subset provided by Theorem \ref{OGM}.
A Dehn filling $\bar{G}$ of $G$, determined by a collection $N_1\trianglelefteq P_1,\dots , N_k\trianglelefteq  P_k$ of normal subgroups, is called \textit{benign with respect to $H$} if:

\begin{enumerate}[{(a)}]
\item\label{benign-1}  The Dehn filling $\pi : G\to \bar{G}$ is  $Z$-long. 

\item\label{benign-3} There is some index $i$  such that $H^{g_i}\cap N_i$ is infinite,  that 
$\pi(H^{g_i})\cap \pi(P_i)$ is finite, and that  $\pi(P_i)=P_i/N_i$ is infinite.

\item \label{benign-4}  The subgroup $\pi(H)$ is relatively quasiconvex with peripherally finite index in $\bar{G}$.

\end{enumerate}
\end{definition}

Condition(\ref{benign-1}) in Definition \ref{benign} guarantees that the Dehn filling $\bar{G}$ is relatively hyperbolic with respect to $\pi(P_i)=P_1/N_1,\dots , \pi(P_k)= P_k/N_k$, and if $\pi(P_i)$ is infinite then $\pi(P_i)$  is a maximal parabolic subgroup in $\bar{G}$ by Theorem \ref{OGM}. The following proposition says that the existence of such a benign Dehn filling of $G$ with respect to $H$ is equivalent to the existence of $P\in\mathbb{P}$, $g\in G$ such that the intersection $H^g\cap P$ is infinite and has infinite index index in $P$.

\begin{proposition}\label{benignproposition}
Let $(G,\mathbb{P})$ be a relatively hyperbolic group where all peripheral subgroups are finitely generated free abelian and let $H$ be a relatively quasiconvex subgroup of $G$. There exists a benign Dehn filling of $G$ with respect to $H$ if and only if $H$ does not have peripherally finite index in $G$.
\end{proposition}

\begin{proof}
Suppose there exists a benign Dehn filling $\pi :G \to \bar{G}$ with respect to $H$.
Let $H^{g_i}$ and $P_i$ be as in part (\ref{benign-3}) of Definition \ref{benign}, so that $H^{g_i}\cap N_i$ is infinite,  $\pi(H^{g_i}) \cap\pi(P_i)$ is finite and $P_i/N_i=\pi(P_i)$ is infinite. 
Since $\pi(H^{g_i}\cap P_i)\le \pi(H^{g_i})\cap \pi(P_i)$, it follows that $ \pi(H^{g_i}\cap P_i)$ has infinite index in $\pi(P_i)$. Hence, $H^{g_i}\cap P_i$ is infinite and has infinite index in $P_i$, and therefore $H$ does not have peripherally finite index in G.

Suppose now that $H$ is relatively quasiconvex but does not peripherally finite index.
List all distinct representatives $\mathbb{D}$ of $H$-conjugacy classes of infinite parabolic subgroups of the form $H\cap P^g$ of $H$.
Group them according to which $P_i$ they come from.
Take a specific $P_i$ and suppose that the subgroups in the above list are $D_1=H\cap P_i^{g_i}, \dots, D_m=H\cap P_i^{g_m}$. Put $U_j=D_j^{g_j^{-1}}$ so that $U_j$ is an infinite finitely generated subgroup of the free abelian group $P_i$.
Take a maximal subcollection $U_1,\dots , U_t$ of $U_j$'s in $P_i$ such that $\bigcup_{i=1}^t U_i$ generates a subgroup $M_i$ of  $P_i$ of infinite index in $P_i$.
If all $U_j$ have finite index in $P_i$, take the subcollection as an empty set.
Then adding any extra $U_y$ to this collection generates a subgroup of finite index in $P_i$.
Now choose a subgroup $N_i$ of finite index in $M_i$ such that $N_1,\dots ,N_k$ is a sufficiently long and $H$-wide filling.

We claim that the Dehn filling $\pi : G\to G/ \langle\langle \bigcup N_i \rangle\rangle$ is benign with respect to $H$. By construction, condition(\ref{benign-1}) in Definition \ref{benign} is satisfied.
By Proposition \ref{GM17-1}, $\pi(H)$ is relatively quasiconvex in $\bar{G}$.
Each of $U_1,\dots , U_t$ is commensurable with a subgroup of $N_i$ and therefore has finite image in $\bar{G}$.
Each $U_y\le P_i$ not from the list $U_1,\dots ,U_t$ has its image $\pi(U_y)$ having finite index in the infinite group $\pi(P_i)=P_i/N_i$.

By Remark  \ref{induced}, the induced peripheral structure on $\pi(H)$ is given exactly by all the infinite groups among $\bar{D}=\pi(D)$ where $D\in \mathbb{D}$. Therefore, $\pi(H)$ has peripherally finite index in $\bar{G}$, and so condition(\ref{benign-4}) in Definition \ref{benign} is satisfied.
Since $H$ does not have peripherally finite index in $G$, there is a $P_i$ such that $N_i$ has infinite index in $P_i$ and that $N_i$ contains a subgroup commensurable with some infinite subgroup $H^a\cap P_i$ of $P_i$ for some $D=H\cap P_i^{a^{-1}}\in \mathbb{D}$.
Then $\pi(H^a\cap P_i)$ is a finite subgroup in the infinite group $\pi(P_i)=P_i/N_i$. By Proposition \ref{finite}, the intersection $\pi(H^a) \cap \pi(P_i)$ is finite, so condition(\ref{benign-3}) in Definition \ref{benign} also holds.
Hence, the filling $\pi : G\to G/ \langle\langle \bigcup N_i \rangle\rangle$ is benign with respect to $H$, as required.\end{proof}

\begin{proof}[Proof of Theorem \ref{main3}(\ref{main3-4})]
Fix a finite set $Z=Z(G)$ as in the conclusion of Theorem \ref{OGM}. 

Let $H$ be an undistorted subgroup of a toral relatively hyperbolic group $(G, \mathbb{P}=\{ P_1, \dots, P_k\})$. 
Since $H$ is undistorted, a result of Hruska \cite{H13} implies that $H$ is relatively quasiconvex in $G$. 

We run the algorithm in Theorem \ref{main3}(\ref{main3-3}), and if the algorithm terminates, then we declare that $H$ is Morse in $G$.

In parallel, we run the following procedure. We start enumerating all plausible candidates for being a benign Dehn filling of $G$ with respect to $H$ as follows. 
Start enumerating all nontrivial elements of the form $h^g$ where $h\in H$ and $g\in G$ and enumerating all elements of $P_1,\dots ,P_k$ and checking if $h^g$ is equal to an element of $P_i$ for some $1\le i\le k$.
Using this check, we enumerate all nontrivial elements $\gamma$ such that $\gamma\in H^g\cap  P_i$ where $g\in G$, $1\le i\le k$.
Using this list, start enumerating all tuples $\tau=(\gamma, Y_1, \dots ,Y_k)$
where $1\ne \gamma\in H^{g_j}\cap P_j$ for some $P_j$ and some $g_j\in G$, and where $Y_1\subseteq P_1,\dots, Y_k\subseteq P_k$ are finite subsets.
Given each such tuple $\tau$, do the following. Put $N_i=\langle Y_i\rangle \trianglelefteq P_i$. Then check if $N_i\cap Z=\emptyset$ for each $i=1,\dots , k$. If not, we discard this tuple $\tau$ and move to the next one. 
Suppose now that $N_i\cap Z=\emptyset$ for each $i=1,\dots,k$, so that  the filling $\pi$ defined by $N_1,\dots ,N_k $ is $Z$-long.
We declare that the Dehn filling $\pi$ a candidate for a benign Dehn filling.  We then run the partial algorithm in Theorem \ref{KMW17-3} on $\pi(H)$. Suppose the algorithm terminates and discovers that $\pi(H)$ is a relatively quasiconvex with peripherally finite index in $\bar{G}$. 
Note that each $P_i/N_i=\pi(P_i)$ is parabolic and thus relatively quasiconvex in $\bar{G}$.  We use the algorithm in Proposition \ref{4.14} to compute a generating set of $\pi(H)^{\pi(g_j)}\cap \pi(P_i)$. If $\pi(H)^{\pi(g_j)}\cap \pi(P_i)$ is finite but $\pi(P_j)=P_j/N_j$ is infinite, then the filling $\pi$ is benign with respect to $H$. Then we terminate the entire procedure and declare that $H$ is not Morse in $G$.


Recall that the subgroup $H$ is relatively quasiconvex in $G$. If $H$ is Morse in $G$, that is, if $H$ has peripherally finite index, then the algorithm in Theorem \ref{main3}(\ref{main3-3}) eventually discovers this fact and terminates. 
Suppose that $H$ is not Morse in $G$, that is $H$ does not have peripherally finite index in $G$.  By Proposition  \ref{benignproposition} this happens if and only if there exists a benign Dehn filling of $G$ with respect $H$. Our second process above will eventually discover such a benign filling and declare that $H$ is not Morse in $G$.
Therefore the above algorithm correctly decides whether or not $H$ is Morse in $G$, as required.  \end{proof}

We use \cite{D03} as background reference to Bowditch boundary of a relatively hyperbolic group in the following proposition.

\begin{proposition}\label{loxodromicelement}
Let $(G,\mathbb{P}$) be a relatively hyperbolic group where $\mathbb{P}$ is a finite collection of finitely generated subgroups. Let $H$ be a relatively quasiconvex subgroup of $G$ with $[G:H]=\infty$. Then there exists a loxodromic element $g\in G$ such that $H\cap \langle g\rangle=\{1\}$.
\end{proposition}

\begin{proof}
For a relatively hyperbolic group $G$ and a relatively quasiconvex subgroup $H$, take the limit set $\Lambda(H)$ in the Bowditch boundary $\partial{G}$. Then $\Lambda(H)$ is a closed subset of $\partial{G}$, and since $H$ is a relatively quasiconvex subgroup of infinite index in $G$, the complement $ \partial{G}\setminus \Lambda(H)$ is nonempty (see Proposition 1.8 in \cite{D03}).
Since the poles of loxodromic elements of $G$ are dense in $\partial{G}$, there exists a loxodromic element $g\in G$ such that $g^\infty\in \partial G\setminus \Lambda(H)$.
Then $\langle g\rangle \cap H=\{1\}$ since otherwise some positive power $g^n$ of $g$ belongs to $H$ and hence $g^\infty\in \partial{H}$.
\end{proof}

\begin{proof}[Proof of Corollary \ref{main5}]
Let $(G,\mathbb{P}$) be a toral relatively hyperbolic group (where $\mathbb{P}$ is possibly empty). Given an undistorted subgroup $H$ of $G$, we do the following.

First, run the algorithm in Theorem \ref{main3}(\ref{main3-4}) for deciding whether $H$ is Morse in $G$.
If $H$ is determined to be non-Morse in $G$, then $H$ has infinite index in $G$.

Suppose now that $H$ turned out to be Morse, that is peripherally of finite index in $G$. We then run in parallel the following two processes:

Run the Todd-Coxeter coset enumeration for detecting finiteness of index of $H$ in $G$. 
In parallel, start enumerating all loxodromic elements $g\in G$. Note that for given a word $g$ in the generators of $G$ we can algorithmically decide whether or not $g$ is loxodromic (see Theorem 5.6 in \cite{O06}).
For each such $g$ the subgroup $\langle g\rangle$ of $G$ is relatively quasiconvex and peripherally of finite index. We then use Proposition \ref{4.14} to compute the subgroup $H\cap \langle g\rangle$. If this subgroup is trivial, we declare that $H$ has infinite index in $G$. If this subgroup is nontrivial (and thus has finite index in $\langle g\rangle$) then we move to the next loxodromic element $g\in G$.

By Proposition \ref{loxodromicelement}, the above procedure decides whether or not $H$ has finite index in $G$.\end{proof}

\section{Groups discriminated by a locally quasiconvex torsion-free hyperbolic group}\label{6}
In this section, we consider a special toral relatively hyperbolic group, that is, a finitely generated group discriminated by a locally quasiconvex torsion-free hyperbolic group. In particular, ordinary limit groups are finitely generated groups discriminated by the free group $F_2$ which is locally quasiconvex torsion-free hyperbolic .

\begin{definition}
We say that a group $G$ is \textit{discriminated} by another group $\Gamma$  if for every finite set $\{g_1, \dots, g_n\}$ of non-trivial elements of $G$ there exists a homomorphism $f : G \rightarrow \Gamma$ such that $f(g_i)\ne 1$ for $i = 1, \dots , n$.
\end{definition}

A word-hyperbolic group $A$ is called \textit{locally quasiconvex} if every finitely generated subgroup of $A$ is quasiconvex. A finitely generated relatively hyperbolic group $A$ is \textit{locally relatively quasiconvex} if every finitely generated subgroup of $A$ is relatively quasiconvex.

Note that if $(G,\mathbb{P})$ is toral relatively hyperbolic and $1\ne g\in G$ then the centralizer $C(g)$ of $g$ in $G$ is the maximal abelian subgroup of $G$ containing $g$. Moreover, if $g$ is loxodromic then $C(g)$ is infinite cyclic if $g$ is parabolic then $C(g)$ is equal to $P^a$ where $P^a$ is the unique conjugate of $P\in \mathbb{P}$ such that $g\in P^a$.


\begin{definition}
Let $A$ be a toral relatively hyperbolic group. An \textit{extension of centralizer} of $A$ is a group presented by 

\centerline{$B=\langle A, t_1, \dots, t_r \,|\, [\,C(g), t_i\,]=[\,t_i, t_j\,]=1 , \,1\le i,\, j \le r\rangle$}
\noindent
where $1\ne g\in A$ and $C(g)$ is the centralizer of $g$ in $A$.
\end{definition}

\begin{proposition}\label{extension}
Let $A$ be a torsion-free hyperbolic groups and let $B$ be an extension of centralizer of $A$. Then $B$ is toral relatively hyperbolic. Moreover, if $A$ is locally relatively quasiconvex then so is $B$. 
\end{proposition}

For the first part of Proposition \ref{extension}, see combination theorems in \cite{D03} and for the second part, see Theorem 3.1 in \cite{BW13}. The following theorem says that for a finitely generated group $G$ discriminated by a torsion-free hyperbolic $\Gamma$, there exists a sequence of centralizer extensions of $\Gamma$ and an embedding of $G$ into a centralizer extension. Theorem \ref{BM16} follows from Theorem B and Theorem C in \cite{KM12}.

\begin{theorem}\label{BM16}
Let $\Gamma$ be a torsion-free hyperbolic group and let $G$ be a finitely generated group discriminated by $\Gamma$.
There exists a sequence centralizer extensions $\Gamma=G_0 <G_1 <\dots<G_n$ where $G_{i+1}$ is an extension of a centralizer of $G_i$ and an embedding $G\hookrightarrow G_n$, and where each $G_i$ is toral relatively hyperbolic. 
\end{theorem}

\begin{lemma}\label{lemma1}
Let $\Gamma$ be a torsion-free locally quasiconvex word-hyperbolic group and let $G$ be a finitely generated group discriminated by $\Gamma$. Then $G$ is a locally relatively quasiconvex toral relatively hyperbolic group (and in particular G is finitely presented).
\end{lemma}

\begin{proof}
Let $\Gamma=G_0\le \dots \le G_n$ be a sequence of extensions of centralizers provided by Theorem \ref{BM16}, where $G \le G_n$.
Note that $G_n$ as in Theorem \ref{BM16} is a locally relatively quasiconvex toral relatively hyperbolic group, by iteratively applying Proposition \ref{extension}. Then the finitely generated group $G$ is relatively quasiconvex in $G_n$ and $G$ is relatively hyperbolic with respect to the induced peripheral structure from $G_n$. Therefore $G$ is toral relatively hyperbolic.

Note that the induced peripheral structure on $G$ from $G_n$ may contain some infinite cyclic peripheral subgroups. However, since infinite cyclic subgroups are word-hyperbolic, they can be dropped from the list of peripheral subgroups, and $G$ will still be relatively hyperbolic with respect to the remaining free abelian (non-cyclic) groups on the list. Thus $G$ is indeed toral relatively hyperbolic.

Let $H$ be a finitely generated subgroup of $G$. Since $H\le G\le G_n$ and $G_n$ is locally relatively quasiconvex, $H$ and $G$ are relatively quasiconvex in $G_n$. Then by Lemma \ref{undistorted}, $H$ and $G$ are undistorted in $G_n$. Therefore $H$ is undistorted in $G$, and this implies that $H$ is relatively quasiconvex in $G$ \cite{H13}.
Thus, $G$ is toral relatively hyperbolic and locally relatively quasiconvex, as required. \end{proof}

\begin{corollary}\label{lemma2}
Let $G$ be a finitely generated group discriminated by a locally quasiconvex torsion-free hyperbolic group.
Then every finitely generated subgroup is undistorted in $G$.
\end{corollary}

\begin{proof}
Let $H$ be a finitely generated subgroup of $G$. 
Lemma \ref{lemma1} implies that $G$ is locally relatively quasiconvex toral relatively hyperbolic, and therefore $H$ is relatively quasiconvex in $G$. Then by Lemma \ref{undistorted} the subgroup $H$ is undistorted in $G$. \end{proof}

\begin{proof}[Proof of Corollary \ref{main4}]

Let $G$ be a finitely generated group discriminated by a locally quasiconvex word-hyperbolic group $\Gamma$. For a finitely generated subgroup $H$ of $G$ given by a finite generating set of $H$, we run the following procedure. 
Note that $G$ is toral relatively hyperbolic and $H$ is undistorted in $G$ by Corollary \ref{lemma2}. For Corollary \ref{main4}(\ref{main4-1}) we run the algorithm in Theorem \ref{main3}(\ref{main3-2}) on $H$ and decide whether or not $H$ is stable in $G$. For Corollary \ref{main4}(\ref{main4-2}), we run the algorithm in Theorem \ref{main3}(\ref{main3-4}) on $H$ and decide whether or not $H$ is Morse in $G$. 
\end{proof}



\begin{thebibliography}{1}








\bibitem[A13]{A13}
T. Aougab, \textit{Uniform Hyperbolicity of the Graphs of Curves}, Geom. Topol. \textbf{17} (2013) 2855–-2875. 




\bibitem[BBKL18]{BKKL18}
 M. Bestvina, K. Bromberg, R. Kent, and C. Leininger,\textit{ Undistorted purely pseudo-Anosov group, }J. Reine Agnew Math. doi:10.1515/crelle-2018-0013 (2018).
 



\bibitem[BC12]{BC12}
J. Behrstock and R. Charney, \textit{Divergence and quasimorphisms of right-angled Artin groups}, Math. Ann. \textbf{352} (2) (2012) 339–-356.



\bibitem[BH95]{BH95}
M. Bestvina and M. Handel, \textit{Train-tracks for surface homeomorphisms}, Topology.  \textbf{34}(1) (1995) 109--140.



\bibitem[BH99]{BH99}
M. Bridson and A. Haefliger, \textit{Metric spaces of non-positive curvature}, Grundlehren der Mathematischen Wissenschaften, 319. Springer-Verlag, Berlin, 1999.


\bibitem[BW13]{BW13}
H. Bigdely and D. Wise, \textit{Quasiconvexity and Relatively Hyperbolic Groups That Split}, Michigan Math. J. \textbf{62} (2013) 387–-406.


\bibitem[BMM16]{BMM16}
J. Birman, D. Margalit, and W. Menasco, \textit{Efficient geodesics and an effective algorithm for distance in the complex of curves}, Math. Ann. \textbf{366}(3-4) (2016) 1253–-1279.


\bibitem[B99]{B99}
B. Bowditch, \textit{Relatively hyperbolic groups}, Preprint, Univ. of Southampton, 1999.


\bibitem[B06]{B06}
B. Bowditch, \textit{Intersection numbers and the hyperbolicity of the curve complex}, J. Reine Angew. Math. \textbf{598} (2006) 105–-129.





\bibitem[B12]{B12}
B. Bowditch, \textit{Uniform hyperbolicity of the curve graphs}, (2012) available at 
\url{http://homepages.warwick.ac.uk/~masgak/papers/uniformhyp.pdf}




\bibitem[BW01]{BW01}
M. Bridson and D. Wise, \textit{Malnormality is undecidable in hyperbolic groups}, Israel J. Math. \textbf{124} (2001) 313–-316.



\bibitem[C07]{C07}
R. Charney, \textit{An introduction to right-angled Artin groups}, Geom. Dedicata. \textbf{125} (2007) 141--158.


\bibitem[CRS14]{CRS14}
M. Clay, K. Rafi, and S. Schleimer, \textit{Uniform hyperbolicity of the curve graph via surgery sequences,} Algebr. Geom. Topol. \textbf{14} (2014)  3325–-3344 


\bibitem[D03]{D03}
 F. Dahmani, \textit{Combination of convergence groups,} Geom. Topol. \textbf{7} (2003) 933-–963.
 





\bibitem[DKL12]{DKL12}
V. Diekert, J. Kausch, and M. Lohrey, \textit{Logspace computations in Coxeter groups and graph groups}, 
Computational and combinatorial group theory and cryptography, 77-–94, 
Contemp. Math., \textbf{582}, Amer. Math. Soc., Providence, RI, 2012. 


\bibitem[DS05]{DS05}
C. Drutu and M. Sapir,\textit{ Tree-graded spaces and asymptotic cones of groups,} Topology. \textbf{44}
(2005) 959–-1058.


\bibitem[DT15]{DT15}
M. Durham and S. Taylor, \textit{Convex cocompactness and
stability in mapping class groups,}  Algebr. Geom. Topol. \textbf{15}(5) (2015) 2839–-2859.


\bibitem[ECHLPT92]{ECHMPT92}
D. Epstein, J. Cannon, D. Holt, S. Levy, M. Paterson, and W. Thurston, \textit {Word Processing in Groups,} Jones and Bartlett, 1992.



\bibitem[FM02]{FM02}
B. Farb and L. Mosher, \textit{Convex cocompact subgroups of mapping class groups},  Geom. Topol. \textbf{6} (2002) 91–-152.


\bibitem[G17]{G17}
A. Genevois, \textit{Hyperbolicities in CAT(0) cube complexes}, arXiv preprint arXiv:1709.08843.




\bibitem[GS91]{GS91}
S. Gersten and H. Short, \textit{Rational Subgroups of Biautomatic Groups}, Ann. Math. \textbf{134} (1991) 125--158.


\bibitem[G]{G}
M. Gromov, \textit{Hyperbolic Groups. In: Essays in Group Theory}, S.M. Gersten (ed.) MSRI series 8, Springer: Berlin Heidelberg New York,  75--263.




\bibitem[GM08]{GM08}
D. Groves and J. Manning, \textit{Dehn filling in relatively hyperbolic groups, } Israel J. Math. \textbf{168} (2008) 317–-429.


\bibitem[GM17]{GM17}
D. Groves and J. Manning, \textit{Quasiconvexity and Dehn filling}, arXiv preprint arXiv:1708.07968v3 (To appear in Amer. J. Math).


\bibitem[H08]{H08}
F. Haglund,  \textit{Finite index subgroups of graph products}, Geom. Dedicata. \textbf{135}(1) (2008) 167–-209.



\bibitem[H05]{H05}
U. Hamenst$\ddot{\textup{a}}$dt, \textit{Word hyperbolic extensions of surface groups}, arXiv preprint arXiv:0505244.


\bibitem[H07]{H07}
U. Hamenst$\ddot{\textup{a}}$dt, \textit{Geometry of the complex of curves and of Teichm$\ddot{u}$ller space, }
Handbook of Teichm\"{u}ller theory, \textbf{1} EMS (2007) 447–-467.


\bibitem[HPW15]{HPW15}
S. Hensel, P. Przytycki, and R. Webb, \textit{1-slim triangles and uniform hyperbolicity for arc graphs and curve graphs,} J. Eur. Math. Soc. (JEMS) \textbf{17}(4) (2015) 755-–762.

 
\bibitem[HM95]{HM95}
S. Hermiller and J. Meier,\textit{ Algorithms and geometry for graph products of groups}, J. Algebra. \textbf{171}(1) (1995)
230–-257.
 
 

\bibitem[H13]{H13}
G. Hruska, \textit{Relative hyperbolicity and relative quasiconvexity for countable groups}, Algebr. Geom. Topol. \textbf{10}(3) (2010) 1807–-1856.


 \bibitem[K96]{K96}
 I. Kapovich, \textit{Detecting Quasiconvexity: Algorithmic Aspects,} Geometric and Computational Perspectives on Infinite Groups, DIMACS Ser. Discrete Math. Theoret. Comput. Sci. \textbf{25} (1996) 91–-99.




\bibitem[KM02]{KM02}
I. Kapovich and A. Myasnikov, \textit{Stallings foldings and the subgroup structure of free groups}, J. Algebra. \textbf{248}(2) (2002) 608–-668.


\bibitem[KWM05]{KWM05}
I. Kapovich, R. Weidmann, and A. Miasnikov, \textit{Foldings, graphs of groups and the membership problem}, Internat. J. Algebra Comput. \textbf{15}(1) (2005) 95–-128.




\bibitem[KL08]{KL08}
A. Kent and C. Leininger, \textit{Shadows of mapping class groups: capturing convex cocompactness,} Geom. Funct. Anal. \textbf{18}(4) (2008) 1270–-1325.











\bibitem[KM12]{KM12}
O. Kharlampovich and A. Myasnikov, \textit{Limits of relatively hyperbolic groups and Lyndon's completions (English summary)}, J. Eur. Math. Soc. (JEMS) \textbf{14}(3) (2012) 659–-680. 





 \bibitem[KMW17]{KMW17}
O. Kharlampovich, A. Miasnikov, and P. Weil,  \textit{Stallings Graphs for Quasi-Convex Subgroups}, J. Algebra. \textbf{488} (2017) 442--483.


\bibitem[KW19]{KW19}
O. Kharlampovich and P. Weil, \textit{On the generalized membership problem in relatively hyperbolic groups}, arXiv preprint arXiv:1908.03525.








\bibitem[K19]{K19}
H. Kim, \textit{Stable subgroups and Morse subgroups in mapping class groups}, Internat. J.
Algebra Comput. \textbf{29}(5) (2019) 893--903.




\bibitem[KK13]{KK13}
S. Kim and T. Koberda, \textit{Embedability between right-angled Artin groups}, Geom. Topol. \textbf{17}(1) (2013) 493–-530.








\bibitem[KK14]{KK14}
S. Kim and T. Koberda,  \textit{The geometry of the curve graph of a right-angled Artin group}, Internat. J.
Algebra Comput. \textbf{24} (2014) 121–-169.



\bibitem[KMT17]{KMT17}
 T. Koberda, J. Mangahas and S. Taylor, \textit{The geometry of purely loxodromic subgroups of right-angled Artin groups}, Trans. Amer. Math. Soc. \textbf{369} (2017) 8179–-8208.



\bibitem[L02]{L02}
J. Leasure, \textit{Geodesics in the complex of curves of a surface}, The University of Texas at Austin, ProQuest Dissertations Publishing, 2002. 3114766.


 \bibitem[LS17]{LS17}
R. Lyndon and P. Schupp, \textit{Combinatorial Group Theory}, 1977, Springer, Berlin, Hei- delberg, New York.



\bibitem[MT16]{MT16}
J. Mangahas and S. Taylor, \textit{Convex cocompactness in mapping class groups via quasiconvexity in right-angled Artin groups}, Proc. London Math. Soc. \textbf{12}(5) (2016) 855–-881.


\bibitem[MM10]{MM10}
J. Manning and E. Martinez-Pedroza, \textit{Separation of relatively quasi-convex subgroups}, Pacific J. Math. \textbf{244}(2) (2010) 309–-334.



\bibitem[M95]{M95}
L. Mosher, \textit{Mapping class groups are automatic}, Ann. of Math. \textbf{142} (1995) 303-384.




\bibitem[O06]{O06}
D. Osin, \textit{Relatively hyperbolic groups: intrinsic geometry, algebraic properties, and algorithmic problems}, Mem. Amer. Math. Soc. \textbf{179}(843) (2006) 1--100.


\bibitem[O07]{O07}
D. Osin, \textit{Peripheral fillings of relatively hyperbolic groups}, Invent. Math. \textbf{167}(2) (2007) 295–-326.



\bibitem[P95]{P95}
P. Papasoglu,\textit{ Strongly geodesically automatic groups are hyperbolic,} Invent. Math. \textbf{121}(2) (1995) 323–-334.



\bibitem[RP77]{RP77}
R. Lyndon and P. Schupp, \textit{Combinatorial group theory. Classics in Mathematics,} Springer-Verlag, Berlin, 2001. Reprint of the 1977 edition.



\bibitem[RST18]{RST18}
J. Russell, D. Spriano, and H. Tran, \textit{Convexity in Hierarchically Hyperbolic Spaces}, arXiv preprint arXiv:1809.09303.





\bibitem[S89]{S89}
H. Servatius, \textit{Automorphisms of graph groups}, J. Algebra. \textbf{126}(1) (1989) 34–-60.



\bibitem[S12]{S12}
K. Shackleton, \textit{Tightness and computing distances in the curve complex}, Geom. Dedicata. \textbf{160} (2012) 243–-259.


\bibitem[T17]{T17}
 H. Tran, \textit{On strongly quasiconvex subgroups, }Geom. Topol. \textbf{23}(3) (2019) 1173–-1235.


\bibitem[W17]{W17}
Y. Watanabe, \textit{Intersection numbers in the curve graph with a uniform constant},  J. Topol. Anal. \textbf{9}(3) (2017) 419-–439.

\bibitem[W15]{W15}
R. Webb, \textit{Combinatorics of tight geodesics and stable lengths}, Trans. Amer. Math. Soc. \textbf{367}(10) (2015) 7323–-7342.




\end{thebibliography}
\end{document}